\documentclass{amsart}
\usepackage[utf8x]{inputenc}
\usepackage{amsmath, amsthm, amssymb}
\usepackage{color}
\usepackage[margin=1.5in]{geometry}
\usepackage{mathtools}
\usepackage{dsfont}

\newtheorem{theorem}{Theorem}[section]
\newtheorem{proposition}{Proposition}[section]
\newtheorem{lemma}{Lemma}[section]
\newtheorem*{remark}{Remark}
\newtheorem{corollary}{Corollary}[section]
\newtheorem{definition}{Definition}[section]
\newcommand{\R}{\mathbb R}

\newcommand{\1}{\mathds{1}}

\newcommand{\osc}{\mbox{osc}}

\numberwithin{equation}{section}

\title{Global wellposedness for the 3D Muskat problem with medium size slope.}
\author{Stephen Cameron}
\address{Courant Institute, New York University, New York, NY 10012}
\email{spc6@cims.nyu.edu}

\begin{document}

\maketitle

\begin{abstract}
We prove the existence and uniqueness of global, classical solutions to the 3D Muskat problem in the stable regime whenever the initial interface has sublinear growth and slope $||\nabla_x f_0||_{L^\infty}< 5^{-1/2}$.  We show under these assumptions that the equation is fundamentally parabolic, satisfying a comparison principle.  Applying the modulus of continuity technique, we show that rough initial data instantly becomes $C^{1,1}$ with the curvature decaying like $O(t^{-1})$.
\end{abstract}

\section{Introduction}

We consider the evolution of the interface between two immiscible, incompressible fluids in a three dimensional porous medium, i.e. the 3D Muskat problem.  When the fluids are of equal viscosity and the physical constants of the system (viscosity, gravity, permeability of the medium) are normalized, the fluid density, velocity, and pressure $(\varrho, u,P)$ satisfy the system of equations
\begin{equation}\label{e:pme}
\left\{\begin{array}{l} \partial_t \varrho + \nabla_X \cdot(u\varrho) = 0, 
\\ u +\nabla_X P + (0,0,\varrho)=0,
\\ \nabla_X \cdot u = 0, 
\\ \varrho(t,X) = \varrho_1 \1_{D(t)}(X) + \varrho_2 \1_{\R^3\setminus D(t)}(X), \end{array}\right.
 \qquad D(t)\subseteq \R^3, \quad (t,X)\in (0,\infty)\times \R^3.
\end{equation}

The stable regime occurs when the fluid domain $D(t)$ is of the form $$D(t) = \{X = (x,x_3)\in \R^3| x_3> f(t,x)\},$$ and the heavier fluid is below the lighter fluid (i.e., $\varrho_1 < \varrho_2)$.  In this case, the dynamics of \eqref{e:pme} can fully be described by the evolution of the interface $f(t,x)$.  Normalizing the mass of the fluids so that $\varrho_2- \varrho_1=4\pi$, it can be shown that the interface $f$ solves 
\begin{equation}
\partial_t f(t,x) = \int\limits_{\R^{2}} \frac{ (\nabla_x f(t,y)- \nabla_x f(t,x))\cdot (y-x)}{((f(t,y)-f(t,x))^2 + |y-x|^2)^{3/2}} dy.
\end{equation}
See \cite{Localwellpose} for a detailed derivation of this equation.  Note that the integral on the right hand side of \eqref{e:fequation} is taken in the principle value sense around $y=x$ and $\infty$.  With a little integration by parts, it can be shown that this equation is equivalent to 
\begin{equation}\label{e:fequation}
\partial_t f(t,x) = \int\limits_{\R^{2}} \frac{f(t,y)-f(t,x) - \nabla_x f(t,x)\cdot (y-x)}{((f(t,y)-f(t,x))^2 + |y-x|^2)^{3/2}} dy,
\end{equation}
which will be more useful for our purposes.  Linearizing \eqref{e:fequation} around a flat interface gives the fractional heat equation, 
\begin{equation}
\partial_t f(t,x) = -(-\Delta_x)^{1/2}f(t,x),
\end{equation}
showing the parabolic nature of the problem for small data, though for large data the nonlinearity becomes highly nontrivial.  

The problem was first proposed by Muskat \cite{Muskat} in a study of the encroachment of water into oil in a tar sand, and in 2D is equivalent to 
a Hele-Shaw cell \cite{Hele}.  The problem was first shown to be locally wellposed in $H^k$ for $k\geq 4$ in \cite{Localwellpose}.  Classical solutions were shown to satisfy maximum principles in both $L^\infty$ \cite{SlopeMax} and $L^2$ \cite{3DSlope}, though neither of these results imply a direct gain in regularity.  

While \eqref{e:fequation} is locally wellposed, global wellposedness is false in general for large initial data.  It was shown in \cite{Breakdown} that wave turning can occur, causing the Rayleigh-Taylor condition to break down in finite time.  That is, there is a smooth solution $f$ to \eqref{e:fequation} and time $T<\infty$ such that 
\begin{equation}
\lim\limits_{t\to T-} ||\nabla_x f||_{L^\infty_x}(t) = \infty, 
\end{equation}
after which the interface between the two fluids is no longer parameterized by the graph of a function.  Once the free boundary leaves the stable regime, it remains smooth for small times due to its initial regularity at time $T$.  However it was shown that this regularity can breakdown in \cite{Blowup} with an example of blowup in $C^4$.  The behavior of the interface once wave turning occurs is complicated in general, as the interface can shift between the stable and unstable regimes multiple times before any regularity breakdown \cite{StabilityShift, StabilityShift2}.  

There has been a large amount of study of what conditions on initial data can guarantee global existence and regularity for solutions to the Muskat problem, particularly in 2D (one dimensional interface).  As solutions to the Muskat problem are preserved under the geometric rescaling $\lambda^{-1}f(\lambda t, \lambda x)$, this has typically taken the form of ``medium-size" upper bounds in scaling invariant norms.  

In 2D, \cite{Globalold} shows a global classical solution exists when the initial data $f_0\in H^3(\R)$ with $||f_0||_1 = ||\ |\xi| \hat{f}_0(\xi)||_{L_\xi^1(\R)} $ less than some explicit constant, which was improved to $\approx 1/3$ in \cite{3DSlope}.  In this case \cite{Decay} proves optimal decay estimates on the norms $||f(t,\cdot)||_s = ||\ |\xi|^s \hat{f}(t,\xi)||_{L_\xi^1}$, matching the estimates for the fractional heat equation.  Under the weaker assumption that $||f'_0||_{L^\infty}<1$, \cite{Globalold} also showed that a maximum principle holds for the slope and global Lipschitz weak solutions exist.  The authors of \cite{Monotonic} were also able show a maximum principle for the slope and the existence of global weak solutions as well, but under the assumption that the initial data $f_0$ was monotonic rather than slope less than 1.  Using a reformulation of \eqref{e:fequation} and a number of Besov space estimates, \cite{H3/2} develops a $\dot{H}^{3/2}$ critical theory for the Muskat problem under a bounded slope assumption.  The authors of \cite{ConstantinMain} made great progress towards proving global regularity by improving the existing continuation criteria from the $C^{2,\delta}$ established in \cite{Localwellpose} to $C^1$.  That is, they proved that if the initial data $f_0\in H^k(\R)$, then the solution $f$ will exist and remain in $H^k$ so long as the slope $f'(t,\cdot)$ remains bounded and uniformly continuous.  

In 3D, most of the medium sized data results in critical spaces can be found in \cite{3DSlope}.  They show that global classical solutions exist when the initial data $f_0$ satisfies $|| \ |\xi| \hat{f}||_{L^1_{\xi}(\R^2)}$ is less than some constant $k_0 \approx 1/5$.  Notably, \cite{ViscosityJump} was able to replicate this result even when the viscosity of the two fluids are distinct, which is the first medium sized data result in the viscosity jump case.  The authors of \cite{3DSlope} also prove a maximum principle for the slope and global weak solutions  whenever the initial slope $||\nabla_x f_0 ||_{L^\infty(\R^2)}$ is suitably bounded.  In the paper they state the theorem for $||\nabla_x f_0||_{L^\infty}< \displaystyle\frac{1}{3}$, but a careful reading of their proof of Theorem 4.1 shows that this holds in fact for the improved bound $||\nabla_x f_0||_{L^\infty} < \displaystyle\frac{1}{\sqrt{5}}$ .  

In our previous work on the 2D case \cite{MuskatMe}, we were able to prove global wellposedness whenever the initial data $f_0\in W^{1,\infty}(\R)$ satisfies 
\begin{equation}
\left(\sup\limits_{x\in \R} f_0'(x)\right)\left(\sup\limits_{y\in \R}-f_0'(y)\right) < 1,
\end{equation}
which is an angular condition that interpolates between the slope less than 1 assumption of \cite{Globalold} and the monotonicity assumption of \cite{Monotonic}.  Under that bound, we showed that the slope satisfied a parabolic equation with maximum principle.  Inspired by the proof of global wellposedness for the critical surface quasi-geostraphic (SQG) equation in \cite{Kiselev}, we then show that this equation generates a Lipschitz modulus of continuity for the slope.  Using these a priori estimates and the continuation criteria established in \cite{ConstantinMain}, we were thus able to get global classical solutions.  

In this work, we extend our results from 2D to 3D proving

\begin{theorem}\label{t:main}
Let $f_0\in \dot{W}^{1,\infty}(\R^{2})$ with $||\nabla_x f_0||_{L^\infty} < \displaystyle\frac{1}{\sqrt{5}}$, and assume that $f_0$ has uniform, integrable sublinear growth.  I.e., there exists some nonnegative function $\Omega: [0,\infty)\to [0,\infty)$ such that 
\begin{equation}
\sup\limits_x \max\limits_{|h|\leq R} |f_0(x+h)-f_0(x)|\leq \Omega(R), \qquad \int\limits_{1}^\infty \frac{\Omega(R)}{R^2}dR <\infty.
\end{equation}
Then there exists a unique, classical solution $f$ to \eqref{e:fequation} with 
\begin{equation}\displaystyle f\in C^{1,\alpha}_{loc}((0,\infty)\times \R^{2})\cap L^\infty_{loc}((0,\infty); C^{1,1}), \qquad \lim\limits_{t\to 0+} ||f(t,\cdot)-f_0(\cdot)||_{L^\infty(\R^{2})} = 0.
\end{equation}
For all times $t$, the solution $f$ satisfies the uniform growth bounds
\begin{equation}
\sup\limits_x \max\limits_{|h|\leq R}|f(t,x+h)-f(t,x)| \leq \Omega(R),
\end{equation}
All directional derivatives $\nabla_x f(t,x)\cdot e$ obey the maximum principle, and 
\begin{equation}\label{e:mainresult}
|\nabla_x f(t,x)-\nabla_x f(t,y)| \leq C(||\nabla_x f_0||_{L^\infty})\frac{|x-y|}{t}, \quad t>0, x,y \in \R^{2}.
\end{equation}
\end{theorem}

Unlike in our previous work \cite{MuskatMe}, we are know able to handle solutions $f$ which are unbounded, so long as they only grow sublinearly.  These new growth bounds, along with the uniqueness and maximum principle for directional derivatives in Theorem \ref{t:main} are a direct consequence of 

\begin{theorem}\label{t:comparison} (Comparison Principle) 

Let $f,g: [0,\infty)\times \R^{d-1}\to \R$ be classical solutions of $d$-dimensional Muskat equation \eqref{e:dequation} with $|| \nabla_x f||_{L^\infty} , ||\nabla_x g ||_{L^\infty} < \displaystyle\frac{1}{\sqrt{2d-1}}$ and $f(0,x)\leq g(0,x)$.  Then $f(t,x)\leq g(t,x)$ for all $t\geq 0$.  
\end{theorem}

%The central piece to the main theorem though is the instantaneous generation of the modulus of continuity $\rho$ for $\nabla f$.  
%
%We prove Theorem \ref{t:main} by deriving a priori estimates for smooth solutions to \eqref{e:fequation} with initial data $f_0\in C^\infty_c(\R^{d-1})$ depending primarily on $||\nabla f_0||_{L^\infty}$.  We prove enough estimates that by approximating in $W^{1,\infty}_{loc}$ with smooth compactly supported initial data, we get solutions $f^\epsilon$ which will converge along subsequences in $C^1_{loc}$ to a solution $f$ solving \eqref{e:fequation} for arbitrary initial data $f_0\in W^{1,\infty}(\R^{d-1})$ with $||\nabla f||_{L^\infty}<\frac{1}{\sqrt{2d-1}}$.

\subsection{Proof outline}

The main strategy is to prove a priori estimates for sufficiently smooth solutions to the Muskat equation \eqref{e:fequation} on the time interval $[0,T]$ depending only on $||\nabla_x f_0||_{L^\infty}$ and $\Omega(R)$.  In particular, the estimates will be independent of the time of existence $T$.  We then use the existing $C^{2,\delta}$ continuation criteria of \cite{Localwellpose} and the vanishing viscosity method from the theory of elliptic equations in order to get global classical solutions for smooth initial data.  Approximating rough, unbounded initial data by smooth compactly supported functions, our a priori estimates will guarantee that we have enough compactness to pass to the limit and get a classical solution to \eqref{e:fequation} which exists for all time.

The most important part of Theorem \ref{t:main} is the $C^{1,1}$ bounds \eqref{e:mainresult}, as the rest of the a priori bounds are a consequence of that.  Following the proof scheme laid out in \cite{Kiselev} proof of global wellposedness for the critical SQG, we do this by showing that the equation \eqref{e:feequation} for the directional derivative $\nabla_x f\cdot e$ generates a Lipschitz modulus of continuity.  That is, we show that there exists a Lipschitz function $\overline{\omega}: [0,\infty)\to [0,\infty)$ depending only on $||\nabla_x f_0||_{L^\infty}$ such that every solution $f$ to the Muskat equation \eqref{e:fequation} satisfies 
\begin{equation}\label{e:omegabarequation}
|\nabla_x f(t,x)-\nabla_x f(t,y)| \leq \overline{\omega}\left(\frac{|x-y|}{t}\right).
\end{equation}
The technique of tailor crafting a modulus of continuity to fit a specific equation was first used in \cite{Kiselev}, but it has since been used on a number of active scalar equations \cite{KiselevBurgers, KiselevSurvey, KiselevIntRearrange}, the 2D Muskat problem \cite{MuskatMe}, and even for geometric flows like fractional mean curvature \cite{FMCFme}.  

In order to explain the main idea behind the proof scheme, let us first consider a simpler example.  Let $u:[0,T]\times \R^{2}\to \R$ solve the drift-diffusion equation
\begin{equation}\label{e:simpleuequation}
\partial_t u(t,x) = b\cdot \nabla_x u(t,x) + \int\limits_{\R^{2}} \delta_h u(t,x) \tilde{K}(h)dh, 
\end{equation}
where $b\in \R^{2}$ is constant and $\tilde{K}$ satisfies 
\begin{equation}
0<\displaystyle\frac{\lambda}{|h|^3}\leq \tilde{K}(h) \leq \frac{\Lambda}{|h|^3}<\infty, \qquad \tilde{K}(h)=\tilde{K}(-h).
\end{equation}
This equation is translation invariant with a comparison principle.  Thus we have that
\begin{equation}
\begin{split}
\sup\limits_{x\in\R^{2}}u(0,x+h)-u(0,x) \leq \omega(|h|) \qquad \Rightarrow \qquad u(0,x)\leq u(0,x+h)+\omega(|h|) \ \forall x\in \R^{2}
\\ \Rightarrow u(t,x) \leq u(t,x+h) +\omega(|h|) , \ \forall (t,x)\in [0,T]\times \R^{2}.
\end{split}
\end{equation}
Thus we get propagation of an arbitrary modulus of continuity $\omega$.  In particular, this implies
\begin{equation}\label{e:simplequalupperbound}
\left\{\begin{array}{ll} u(t_0,x+h)-u(t_0,x) \leq \omega(|h|), & \forall x,h\in \R^{2}, \\ u(t_0,\xi/2)-u(t_0,-\xi/2) = \omega(|\xi|), & \xi\in \R^{2}  \end{array}\right. \Rightarrow \qquad \frac{d}{dt}\left(u(t,\xi/2)-u(t,-\xi/2)\right)\bigg|_{t=t_0} \leq 0,
\end{equation}
This implication though only relies on the nonnegativity $\tilde{K}(h)\geq 0$, not the full uniform ellipticity.  A more refined argument would in fact prove a strictly negative upper bound
\begin{equation}\label{e:simplequantupperbound}
\frac{d}{dt}\left(u(t,\xi/2)-u(t,-\xi/2)\right)\bigg|_{t=t_0} \leq C_{\lambda}[\omega](\xi)< 0.  
\end{equation}
Thus we have a strictly negative, quantitative upperbound depending on the modulus $\omega$, ellipticity constants, and the crossing point $\xi$.  This strict negativity in fact implies that our solution $u$ actually regularizes, allowing us to improve the modulus of continuity $\omega$ over time.  

The equation \eqref{e:feequation} that our directional derivative $f_e$ solves is not nearly as nice as \eqref{e:simpleuequation}.  However, because the upperbound \eqref{e:simplequantupperbound} is quantitative, there is hope that we may still be able to prove the same result for $f_e$ so long as we can quantitatively bound how far \eqref{e:feequation} is from a translation invariant, symmetric equation with a comparison principle like \eqref{e:simpleuequation}, and then choose the modulus $\omega$ correctly.

\subsection{Organization of the paper}
The rest of the paper is organized as follows:
\begin{itemize}

\item In Section \ref{s:maxprinciple} we differentiate \eqref{e:fequation} to derive the equation for the directional derivatives $f_e = \nabla_x f\cdot e$, showing that they solve a uniformly parabolic equation \eqref{e:feequation} with a maximum principle when $||\nabla_x f||_{L^\infty(\R^2)}<\frac{1}{\sqrt{5}}$.  

\item In Section \ref{s:breakthrough} we repeat the breakthrough argument of \cite{Kiselev}, reducing the proof of generation of a modulus modulus of continuity for $\nabla_x f$ down to proving an inequality.  

\item In Section \ref{s:modulusestimates} we derive asymmetry, continuity, and diffusive bounds on our drift term and elliptic kernel $K$ for \eqref{e:feequation} depending on the modulus of continuity of $\nabla_x f$. 

\item In Section \ref{s:modulusinequality} we then use these bounds to bound the time derivative of our slope, and then apply them to a specific modulus of continuity, proving propagation of regularity in the process.  

\item In Section  \ref{s:modulusrho} we combine the propagation of regularity along with scale invariance to get the generation of the modulus $\overline{\omega}$ such that \eqref{e:omegabarequation} holds.  

\item In Section \ref{s:comparison} we prove the comparison principle for the Muskat equation \eqref{e:fequation} along with a number of corollaries arising from that.  In particular, we show uniqueness for classical solutions and propagation of sublinear growth bounds when $||\nabla_x f||_{L^\infty(\R^2)} < \frac{1}{\sqrt{5}}$.  

\item In Section \ref{s:regularitytime} we use our $C^{1,1}$ estimates and growth bounds in order to prove a few estimates on regularity in time, guaranteeing compactness in $C^1$.  

\item Finally in Section \ref{s:globalexistence} we use the vanishing viscosity method, our a priori estimates, and the $C^{2,\delta}$ continuation criteria of \cite{Localwellpose} in order to prove that there exist global, classical solutions.

\end{itemize}

\subsection{Notation}

For $h\in \R^2$, we let $\delta_h$ denote the partial difference operator 
\begin{equation}
\delta_h g(x):= g(x+h)-g(x).
\end{equation}
For $e\in S^1$, we let $g_e$ denote the directional derivative
\begin{equation}
g_e(x) = \nabla_x g(x)\cdot e.
\end{equation}
We say that a quantity $A \lesssim B$ if $A$ is bounded above by $B$, up to a multiplicative constant depending only on the initial slope $||\nabla_x f||_{L^\infty}$.  That is, 
\begin{equation}
A\lesssim B  \quad \iff \quad A \leq C(||\nabla_x f_0||_{L^\infty}) B.  
\end{equation}

%%%%%%%%%%%%%%%%%%%%%%%%%%%%%%%%%%%%%%%%%%%%%%%%%%%%%%%%%%%%%%%%%%%%%%

\section{Maximum Principle for $\nabla_x f$}\label{s:maxprinciple}

Let $h\in \R^2$, and $\delta_h$ denote the finite difference operator 
\begin{equation}
\delta_h g(x):= g(x+h) - g(x).
\end{equation}
Then making a simple change of variables, the Muskat equation can be written in the form
\begin{equation}\label{e:fequation2}
\partial_t f(t,x) = \int\limits_{\R^{2}} \frac{\delta_hf(t,x) - \nabla f(t,x)\cdot h}{(\delta_h f(x)^2 + |h|^2)^{3/2}} dh,
\end{equation}
Taking $e\in S^1$ and differentiating \eqref{e:fequation2} 
with respect to $e$ gives us 
\begin{equation}\label{e:feequation}
\begin{split}
\partial_t f_e(t,x) &= \int\limits_{\R^{2}} \frac{-\nabla f_e(t,x)\cdot h}{(\delta_hf(t,x)^2 + |h|^2)^{3/2}}dh 
 + \int\limits_{\R^{2}} \delta_h f_e(t,x) K(t.x,h)dh,
\end{split}
\end{equation}
where 
\begin{equation}\label{e:Kdefn}
K(t,x,h) =  \frac{1}{(\delta_h f(t,x)^2 + |h|^2)^{3/2}} \left(1 - 3\frac{\delta_hf(t,x)(\delta_hf(t,x)-\nabla f(t,x)\cdot h)}{\delta_hf(t,x)^2+|h|^2}\right).
\end{equation}
Note that if $||\nabla_x f||_{L^\infty} \leq \displaystyle\frac{1}{\sqrt{5}}$, then 
\begin{equation}
1 - 3\frac{\delta_hf(t,x)(\delta_hf(t,x)-\nabla f(t,x)\cdot h)}{\delta_hf(t.x)^2+|h|^2} \geq 1 - 3\frac{2(5)^{-1}}{(5)^{-1}+1} = 1-3\frac{2}{6}=0.
\end{equation}
Hence, $K(t,x,h)\geq 0$ whenever $||\nabla_x f||_{L^\infty} \leq\displaystyle \frac{1}{\sqrt{5}}$ with 
\begin{equation}
\frac{\lambda }{|h|^2}\leq K(t,x,h)\leq \frac{\Lambda}{|h|^2},
\end{equation}
whenever $||\nabla_x f||_{L^\infty(\R^2)} < \displaystyle \frac{1}{\sqrt{5}}$. 

Taking advantage of this, under this initial slope bound we get a maximum principle for $\nabla_x f$ for sufficiently smooth solutions.  

\begin{proposition}\label{p:maximumprinciple}
Let $f:[0,T]\times \R^{2}\to \R$ be a sufficiently smooth solution to \eqref{e:fequation} such that all directional derivative $f_e$ for $e\in S^{1}$ are classical solutions to \eqref{e:feequation}.  Then if $||\nabla_x f_0||_{L^\infty}<\displaystyle\frac{1}{\sqrt{5}}$, then the kernel $K$ defined in \eqref{e:Kdefn} is uniformly elliptic.  That is, 
\begin{equation}\label{e:uniformellipticity}
\frac{\lambda }{|h|^2}\leq K(t,x,h)\leq \frac{\Lambda}{|h|^2},
\end{equation} 
with ellipticity constants $\lambda, \Lambda $ depending only on $||\nabla_x f_0||_{L^\infty}$.  In particular, directional derivatives obey the maximum principle with 
\begin{equation}
\sup\limits_{x\in\R^{2}} f_e(t,x)\leq \sup\limits_{x\in \R^{2}} f_e(0,x), \qquad \forall t\in [0,T],
\end{equation}
and hence 
\begin{equation}
 ||\nabla_x f||_{L^\infty}(t)\leq ||\nabla_x f_0||_{L^\infty}.
\end{equation}
\end{proposition}

\begin{proof}
Let $f:[0,T]\times \R^{2}\to \R$ be a sufficiently smooth solution to \eqref{e:fequation} with $||\nabla f_0||_{L^\infty}< \displaystyle\frac{1}{\sqrt{5}}$.  Fix some arbitrary constant $B$ with 
\begin{equation}
||\nabla f_0||_{L^\infty}<B<\frac{1}{\sqrt{5}}.
\end{equation}
Defining the time $t^*$ by 
\begin{equation}
t^* = \sup\left\{t\in [0,T]: \  ||\nabla f||_{L^\infty}(s)\leq B \qquad \forall s\leq t \right\}, 
\end{equation}
then we necessarily have that $t^*>0$, and in fact we shall show that $t^*=T$.

For any time $t<t^*$, we have that 
\begin{equation}
||\nabla f||_{L^\infty}(t)\leq B.  
\end{equation}
Fix $e\in S^{1}$, arbitrary.  Then for any $t<t^*$ and $x,h\in \R^{2}, h\not = 0$, we have that 
\begin{equation}
0<\left(\frac{1}{(B^2+1)^{3/2}}\right)\left(1-3\frac{2B^2}{B^2+1}\right)\frac{1}{|h|^3} \leq K(t,x,h)\leq \left(1+3\frac{2B^2}{B^2+1}\right)\frac{1}{|h|^3} <\infty.
\end{equation}
Thus as $f_e$ is a classical solution to a uniformly elliptic equation on the time scale $[0,t^*]$, we have that the maximum principle holds:
\begin{equation}
\sup\limits_{x\in \R^{2}} f_e(t,x)\leq \sup\limits_{x\in \R^{2}}f_e(0,x)\leq ||\nabla f_0||_{L^\infty}<B, \qquad \forall t\in [0,t^*].
\end{equation}
Taking the supremum in $e\in S^{1}$, we thus have
\begin{equation}
||\nabla f||_{L^\infty}(t)\leq ||\nabla f_0||_{L^\infty} <B, \qquad \forall t\in [0,t^*].
\end{equation}
By the definition of the time $t^*$ and the continuity of $||\nabla f||_{L^\infty}(t)$ as a function of time, we thus have that in fact $t^*=T$.  Hence, every directional derivative $f_e$ obeys the maximum principle on the time interval $[0,T]$, and solves a uniformly elliptic equation with ellipticity constants
\begin{equation}
0<\lambda = \left(\frac{1}{(||\nabla f_0||_{L^\infty}^2+1)^{3/2}}\right)\left(1 - 3\frac{2||\nabla f_0||_{L^\infty}^2}{||\nabla f_0||_{L^\infty}^2+1} \right)\leq 1 + 3\frac{2||\nabla f_0||_{L^\infty}^2}{||\nabla f_0||_{L^\infty}^2+1} = \Lambda <\infty.
\end{equation}
\end{proof}

\begin{remark}
We note that maximum principle for the directional derivatives $f_e$ also follows directly from the comparison principle proven in Section \ref{s:comparison}.  As the uniform ellipticity of the kernel $K$ is vital for our later arguments though, we chose to prove these two results separately.  
\end{remark}

%Since this is a uniformly elliptic equation, we have that $f_e(t,x)$ has the maximum principle.  Thus if $f$ is a classical solution to \eqref{e:fequation} with $||\nabla f_0||_{L^\infty} < \displaystyle \frac{1}{\sqrt{2d-1}}$, then 
%\begin{equation}
%\max\limits_x f_e(t,x) \leq \max\limits_x f_e(0,x) \ \ \forall e\in S^{d-1} \quad \Rightarrow \quad ||\nabla f(t,\cdot)||_{L^\infty}\leq ||\nabla f_0 ||_{L^\infty}.
%\end{equation}
%
%Thus so long as our initial data satisfies $||\nabla f_0||_{L^\infty}< \frac{1}{\sqrt{2d-1}}$, we will have that $K(t,x,h)$ is uniformly elliptic for all $(t,x)\in [0,\infty)\times \R^{d-1}$.  

%%%%%%%%%%%%%%%%%%%%%%%%%%%%%%%%%%%%%%%%%%%%%%%%%%%%%%%%%%%%%%%%%%%%%%

\section{Breakthrough Argument}\label{s:breakthrough}

%Let $f:[0,T]\times \R^2\to \R$ be a smooth, decaying solution of the Muskat equation \eqref{e:fequation}.  In particular, we assume that the gradient of our solution $\nabla_x f$ is $C^1$ in time, $C^2$ in space with 
%\begin{equation}
%\lim\limits_{|x|\to \infty} \nabla_x f(t,x), D^2_x f(t,x) = 0, \qquad \text{ uniformly for } t\in [0,T].
%\end{equation}

%Assume that we have a global solution $f:[0,\infty)\times \R^{d-1}\to \R$ with $\nabla f$ $C^1$ in time and $C^{2}$ in space and that $\lim\limits_{|x|\to \infty} \nabla f(t.x) =0.$  To justify this, we take $f_0 \in C^\infty_c(\R^{d-1})$ and let $f$ be a solution to the regularized system
%\begin{equation}
%\partial_t f(t,x) = \epsilon \Delta f(t,x) + \int\limits_{\R^{d-1}} \frac{f(t,y)-f(t,x) - \nabla f(t,x)\cdot (y-x)}{((f(t,y)-f(t,x))^2 + |y-x|^2)^{d/2}} dy.
%\end{equation}
%The $\epsilon \Delta$ should guarantee the global smooth, decaying solution.  And the added viscosity doesn't hurt any of our later estimates, so we can just ignore it and derive the apriori estimate for the original equation.  

Let $f:[0,T]\times \R^2\to \R$ be a smooth, decaying solution of the Muskat equation \eqref{e:fequation}, and fix some concave function $\overline{\omega}:[0,\infty)\to [0,\infty)$ with $\overline{\omega}(0)=0$.  Then since $f$ is smooth and decays at $\infty$, we then have that 
\begin{equation}
|\nabla_x f(t,x)-\nabla_x f(t,y)| < \overline{\omega}\left(\frac{|x-y|}{t}\right), \qquad \text{ for all } x\not = y\in \R^2, \text{ and for all times } t \text{ sufficiently small}.
\end{equation}
Suppose that $\nabla f(t,\cdot)$ has modulus of continuity $\overline{\omega}(\cdot/ t)$ for all $t<t_0$ for some time $t_0<T$.  Then by continuity,
\begin{equation}
|\nabla_x f(t_0,x)-\nabla_x f(t_0,y)| \leq \overline{\omega}\left(\frac{|x-y|}{t_0}\right), \quad \forall x\not = y\in \R^{2}.
\end{equation}
We first prove that if we have the strict inequality $|\nabla_x f(t_0,x)-\nabla_x f(t_0,y)| < \overline{\omega}\left(|x-y|/t_0\right)$, then $\nabla_x f(t_0,\cdot)$ will have modulus $\overline{\omega}(\cdot/ t)$ for $t\leq t_0+\epsilon$.

\begin{lemma}
Let $f\in C([0,T]; C^3_0(\R^{2}))$, and $t_0\in (0,T)$.  Suppose that $f(t_0,\cdot)$ satisfies
\begin{equation}
|\nabla_x f(t_0,x)-\nabla_x f(t_0,y)| < \overline{\omega}\left(\frac{|x-y|}{t_0}\right), \forall x\not = y\in \R^{2},
\end{equation}
for some Lipschitz modulus of continuity $\overline{\omega}$ with $\overline{\omega}''(0) = -\infty$.  Then
\begin{equation}
|\nabla_x f(t_0+\epsilon, x) - \nabla f(t_0+\epsilon, y)| < \overline{\omega}\left(\frac{|x-y|}{t_0+\epsilon}\right), \forall x\not = y\in \R,
\end{equation}
for all $\epsilon>0$ sufficiently small.
\end{lemma}

\begin{proof}
To begin, note that for any compact compact subset $K\subset \R^{4}\setminus \{(x,x)|x\in \R^{2}\}$,
\begin{equation}
\begin{split}
|\nabla_x f(t_0,x)&-\nabla_x f(t_0,y)| < \overline{\omega}\left(\frac{|x-y|}{t_0}\right) \quad \forall (x,y)\in K \quad
\\&  \Rightarrow \quad \nabla_x f(t_0+\epsilon, x)-\nabla_x f(t_0+\epsilon,y) < \overline{\omega}\left(\frac{|x-y|}{t_0+\epsilon}\right) \quad \forall (x,y)\in K,
\end{split}
\end{equation}
for $\epsilon >0$ sufficiently small by uniform continuity.  So, we only need to focus on pairs $(x,y)$ that are either close to the diagonal, or that are large.

To handle $(x,y)$ near the diagonal, we start by noting that $f(t_0,\cdot)\in C^3(\R^{2})$ and $\overline{\omega}''(0)=-\infty$.  Thus for every $x$ we get that
\begin{equation}
|D^2_xf(t_0,x)| := \max_{e\in S^1} \partial_{ij}^2f(t,x)e_i e_j < \frac{\overline{\omega}'(0)}{t_0}.
\end{equation}
Since $f\in C([0,T]; C^3_0(\R^2))$, $D_x^2f_(t_0,x)\to 0$ as $|x|\to \infty$.  Thus we can take the point where $\max\limits_x ||D_x^2f(t_0,x)||$ is achieved to get that
\begin{equation}
||D_x^2f(t_0,\cdot)||_{L^\infty} < \frac{\overline{\omega}'(0)}{t_0}.
\end{equation}
By continuity of $D_x^2f$, we thus have 
\begin{equation}
||D_x^2f(t_0+\epsilon,\cdot)||_{L^\infty} < \displaystyle\frac{\overline{\omega}'(0)}{t_0+\epsilon},
\end{equation}
for $\epsilon>0$ sufficiently small.  Hence,
\begin{equation}\label{e:smallxypropogation}
|\nabla_x f(t_0+\epsilon,x)-\nabla_x f(t_0+\epsilon,y)| < \overline{\omega}\left(\frac{|x-y|}{t_0+\epsilon}\right), \quad |x-y|<\delta,
\end{equation}
for $\epsilon, \delta$ sufficiently small.

Now let $R_1,R_2>0$ be such that
\begin{equation}
\overline{\omega}\left(\frac{R_1}{t_0+\epsilon}\right) > \osc_{\R^{2}} \nabla_x f(t_0+\epsilon,\cdot),
\end{equation}
and that $|x|>R_2$ implies
\begin{equation}
|\nabla_x f(t_0+\epsilon,x)| < \frac{1}{2}\overline{\omega}\left(\frac{\delta}{T+\epsilon}\right)
\end{equation}
for $\epsilon>0$ sufficiently small.
Taking $R = R_1+R_2$, it's easy to check that $|x|>R$ implies that
\begin{equation}
|\nabla_x f(t_0+\epsilon,x) - \nabla_x f(t_0+\epsilon,y)| < \overline{\omega}\left(\frac{|x-y|}{t_0+\epsilon}\right), \quad \forall y\not = x.
\end{equation}

Finally, taking $K= \{ (x,y)\in \R^{4}: |x-y|\geq \delta, |x|,|y|\leq R\}$, we're done.

\end{proof}

Thus by the lemma, if $\nabla_x f$ was to lose its modulus after time $t_0$, we must have that there exist $x\not = y\in \R^2$ with
\begin{equation}
|\nabla_x f(t_0,x) - \nabla_x f(t_0,y)| = \overline{\omega}\left(\frac{|x-y|}{t_0}\right).
\end{equation}
and hence a direction $e\in S^{1}$ such that 
\begin{equation}\label{e:modulusequality}
f_e(t_0,x) - f_e(t_0,y)| = \overline{\omega}\left(\frac{|x-y|}{t_0}\right).
\end{equation}

We will show for a smooth solution $f$ of \eqref{e:fequation} and the correct choice of $\overline{\omega}$ that in this case
\begin{equation}\label{e:finalinequality}
\frac{d}{dt} \left(f_e(t,x) - f_e(t,y)\right)\bigg|_{t=t_0} < \frac{d}{dt}\left(\overline{\omega}\left(\frac{|x-y|}{t}\right)\right)\bigg|_{t=t_0},
\end{equation}
contradicting the fact that $f_e$ had modulus $\overline{\omega}(\cdot /t)$ for time $t<t_0$.

Thus we just need to prove \eqref{e:finalinequality} to complete the proof of the generation of modulus of continuity, completing the $C^{1,1}$ estimate \eqref{e:mainresult} of Theorem \ref{t:main}.

%%%%%%%%%%%%%%%%%%%%%%%%%%%%%%%%%%%%%%%%%%%%%%%%%%%%%%%%%%%%%%%%%%%%%%

\section{Modulus Estimates}\label{s:modulusestimates}
Let $f $ be a sufficiently smooth solution to the 3-dimensional Muskat equation \eqref{e:fequation} with $||\nabla f_0||_{L^\infty(\R^2)} < \displaystyle\frac{1}{\sqrt{5}}$.  

Suppose that $\omega:[0,\infty)\to [0,\infty)$ is such that 
\begin{equation}\label{e:omegabound}
\left\{\begin{array}{ll} 
\omega(0)=0, \omega''(r)<0, & \forall r\in (0,\infty)
\\|\delta_h \nabla_x f(t_0,x)| \leq \omega(|h|), & \forall x,h\in \R^{2}, \\ f_e(t_0,\xi/2)-f_e(t_0,-\xi/2) = \omega(|\xi|), & e\in S^{1}, \xi\in \R^{2} \text{ fixed,} \end{array}\right.
\end{equation}
where $\delta_h$ denotes the partial difference operator $\delta_h g(x) = g(x+h)-g(x)$.  

Our goal is to derive an upper bound on 
\begin{equation}\label{e:timediffboundgoal}
\frac{d}{dt}\left(f_e(t,\xi/2)-f_e(t,-\xi/2)\right)\bigg|_{t=t_0},
\end{equation}
in terms of the modulus $\omega$ and the initial slope $||\nabla_x f_0||_{L^\infty}$.  In order to do this, we first need to derive estimates on the asymmetry and $x$-dependence of the drift and kernel $K$ from \eqref{e:feequation} in order to bound how far \eqref{e:feequation} is from being a simple drift-diffusion equation like \eqref{e:simpleuequation}.  We will then use these to bound the difference in diffusions
\begin{equation}
\int\limits_{\R^{2}} \delta_h f_e(\xi/2)K(\xi/2, h) -\delta_h f_e(-\xi/2)K(-\xi/2, h)dh,
\end{equation}
at the end of this section, and finally give an upper bound on \eqref{e:timediffboundgoal} in Lemma \ref{l:timederivbound}.

\subsection{Asymmetry Bounds}
\begin{lemma}\label{l:asymmetry}
Let $f$ satisfy \eqref{e:omegabound} and the kernel $K$ be as in \eqref{e:Kdefn}.  Then the drift and kernel $K$ satisfy the pointwise asymmetry bounds
\begin{equation}
\begin{split}
\bigg| \frac{1}{(\delta_h f(x)^2 + |h|^2)^{3/2}} - \frac{1}{(\delta_{-h}f(x)^2 + |h|^2)^{3/2}} dh \bigg| \lesssim \frac{\omega(|h|)}{|h|^3}, 
\\ |K(x,h)-K(x,-h)|\lesssim \frac{\omega(|h|)}{|h|^3}.
\end{split}
\end{equation}
In particular, integrating over $|h|<|\xi|$
\begin{equation}
\begin{split}
\int\limits_{|h|<|\xi|} \bigg|\frac{|h|}{(\delta_h f(\pm \xi/2)^2 + |h|^2)^{3/2}} - \frac{|h|}{(\delta_{-h}f(\pm \xi/2)^2 + |h|^2)^{3/2}} \bigg| dh \lesssim \int\limits_{0}^{|\xi|}\frac{\omega(r)}{r}dr,
\\ \int\limits_{|h|<|\xi|} |h| |K(\pm \xi/2, h)-K(\pm \xi/2, -h)|dh \lesssim \int\limits_{0}^{|\xi|}\frac{\omega(r)}{r}dr.
\end{split}
\end{equation}
\end{lemma}
\begin{proof}
To begin, note that as 
\begin{equation}
|\delta_hf(x) + \delta_{-h}f(x) |  = \bigg|\left(\int\limits_0^1 \nabla f(x+sh)-\nabla f(x+(s-1)h)ds\right)\cdot h\bigg| \leq \omega(|h|)|h|, 
\end{equation}
it follows that 
\begin{equation}\label{e:driftasymbound}
\bigg|\frac{1}{(\delta_h f(x)^2 + |h|^2)^{3/2}} - \frac{1}{(\delta_{-h} f(x)^2 + |h|^2)^{3/2}}\bigg|\lesssim \frac{|\delta_h f(x)+\delta_{-h}f(x)|}{|h|^{4}} \lesssim \frac{\omega(|h|)}{|h|^3}.
\end{equation}

As we similarly have that $|\delta_hf(x) - \nabla f(x)\cdot h| \leq \omega(|h|)|h|$, it follows that 
\begin{equation}\label{e:Kasymbound}
\bigg|\frac{\delta_hf(x)(\delta_hf(x)-\nabla f(x)\cdot h)}{\delta_hf(x)^2+|h|^2} \bigg| \lesssim \omega(|h|).
\end{equation}
Recalling the equation for  $K$
\begin{equation}
K(x,h)  = \frac{1}{(\delta_h f(x)^2 + |h|^2)^{3/2}}\left(1- 3\frac{\delta_hf(x)(\delta_hf(x)-\nabla f(x)\cdot h)}{\delta_hf(x)^2+|h|^2} \right),
\end{equation}
we see by combining \eqref{e:driftasymbound} and \eqref{e:Kasymbound} that
\begin{equation}
|K(x,h)-K(x,-h)|\lesssim \frac{\omega(|h|)}{|h|^3},
\end{equation}
as well.  
\end{proof}

\subsection{Continuity Bounds}

\begin{lemma}\label{l:continuity}
Let $f$ satisfy \eqref{e:omegabound} and the kernel $K$ be as in \eqref{e:Kdefn}.  Then the drift and kernel $K$ satisfy the pointwise continuity bounds
\begin{equation}
\begin{split}
\bigg|\frac{1}{(\delta_hf(\xi/2)^2+|h|^2)^{3/2}} - \frac{1}{(\delta_hf(-\xi/2)^2+|h|^2)^{3/2}} \bigg|\lesssim \min\left\{ \frac{\omega(|\xi|)}{|h|^3},  \frac{\omega(|h|)|\xi|}{|h|^{4}}\right\} ,
\\ |K(\xi/2, h)-K(-\xi/2, h)|\lesssim \frac{\omega(|\xi|)}{|h|^3}.
\end{split}
\end{equation}
In particular, 
\begin{equation}
\int\limits_{|h|>|\xi|} \bigg| \frac{|h|}{(\delta_hf(\xi/2)^2+|h|^2)^{3/2}} - \frac{|h|}{(\delta_hf(-\xi/2)^2+|h|^2)^{3/2}} \bigg|dh \lesssim |\xi| \int\limits_{|\xi|}^\infty \frac{\omega(r)}{r^2}dr,
\end{equation}
\end{lemma}

\begin{proof}
In order to bound the $x$-dependence of our kernel and drift, we first need to note that 
\begin{equation}
|\delta_hf(\xi/2) -\delta_h f(-\xi/2)| = \bigg|\left(\int\limits_0^1 \nabla_x f(\xi/2 +sh)-\nabla_x f(-\xi/2+sh)ds\right)\cdot h \bigg|\leq \omega(|\xi|)|h|,
\end{equation}
and 
\begin{equation}
\begin{split}
\bigg|\delta_hf(\xi/2) - \delta_h f(-\xi/2) \bigg|&=\bigg| (f(\xi/2+h)-f(\xi/2)) - (f(-\xi/2+h)-f(-\xi/2)) \bigg|
\\&=\bigg| \left(\int\limits_0^1 \nabla_x f(h+(s-1/2)\xi) - \nabla_x f((s-1/2)\xi)ds\right)\cdot \xi\bigg| \leq \omega(|h|)|\xi|.
\end{split}
\end{equation}

Hence it follows that 
\begin{equation}
\bigg|\frac{1}{(\delta_hf(\xi/2)^2+|h|^2)^{3/2}} - \frac{1}{(\delta_hf(-\xi/2)^2+|h|^2)^{3/2}} \bigg|\lesssim \frac{|\delta_h f(\xi/2)-\delta_hf(-\xi/2)|}{|h|^{4}} \lesssim \min\left\{ \frac{\omega(|\xi|)}{|h|^3},   \frac{\omega(|h|)|\xi|}{|h|^{4}}\right\}.
\end{equation}

As we already have that 
\begin{equation}
\bigg|K(x,h)-\frac{1}{(\delta_hf(x)^2+|h|^2)^{3/2}} \bigg| \lesssim \frac{\omega(|h|)}{|h|^{3}},
\end{equation}
we get immediately that for $|h|\leq |\xi|$ 
\begin{equation}
\bigg|K(\xi/2, h)-K(-\xi/2, h)\bigg|\lesssim \frac{\omega(|\xi|)}{|h|^3}.
\end{equation}

For $|h|\geq |\xi|$, we note that 
\begin{equation}
\begin{split}
\bigg|\frac{\delta_hf(\xi/2)(\delta_hf(\xi/2)-\nabla f(\xi/2)\cdot h)}{\delta_hf(\xi)^2+|h|^2} &- \frac{\delta_hf(-\xi/2)(\delta_hf(-\xi/2)-\nabla f(-\xi/2)\cdot h)}{\delta_hf(-\xi/2)^2+|h|^2}\bigg| 
\\ &\lesssim \frac{|\delta_h f(\xi/2)-\delta_h f(-\xi/2)|}{|h|} + \frac{|(\nabla f(\xi/2)-\nabla f(-\xi/2))\cdot h|}{|h|} 
\\ & \lesssim \frac{\omega(|h|)|\xi|}{|h|}+\omega(|\xi|).
\end{split}
\end{equation}
As by assumption $\omega$ is concave, the function $r \to \displaystyle\frac{\omega(r)}{r} = \frac{\omega(r)-\omega(0)}{r}$ is non increasing.  Hence, 
\begin{equation}
\frac{\omega(|h|)}{|h|}\leq \frac{\omega(|\xi|)}{|\xi|}, \qquad |h|\geq |\xi|.
\end{equation}
Thus for $|h|\geq |\xi|$, we have that 
\begin{equation}
|K(\xi/2, h)-K(-\xi/2, h)|\lesssim \frac{\omega(|h|)|\xi|}{|h|^{4}} + \frac{\omega(|\xi|)}{|h|^{3}}\lesssim \frac{\omega(|\xi|)}{|h|^3}.
\end{equation}
\end{proof}

\subsection{Diffusive Bounds}

\begin{lemma}\label{l:diffusion}
Let $f$ satisfy \eqref{e:omegabound} and the kernel $K$ be as in \eqref{e:Kdefn} satisfying the uniform ellipticity bounds \eqref{e:uniformellipticity}.  Then 
\begin{equation}
\begin{split}
\int\limits_{\R^{2}} \delta_h f_e(\xi/2)K(\xi/2, h) - &\delta_h f_e(-\xi/2)K(-\xi/2, h) dh - \lambda \int\limits_{\R^{2}} \frac{\delta_h f_e(\xi/2)-\delta_{h}f_e(-\xi/2)}{|h|^3} dh
\\& \lesssim \omega'(|\xi|)\int\limits_{0}^{|\xi|} \frac{\omega(r)}{r}dr + \omega(|\xi|)\int\limits_{|\xi|}^\infty \frac{\omega(|\xi|+r)-\omega(|\xi|)}{r^2}dr.
\end{split}
\end{equation}
\end{lemma}

%\section{Time Derivative Bounds}
%
%We remark that everything in sections 4, 5, and 6 relies only on the estimates proven in Lemmas \ref{l:asymmetry} and \ref{l:continuity}, as well as the uniform ellipticity of the kernel $K$.  
%
%\begin{lemma}
%Let $f$ be a solution of \eqref{e:fequation} and satisfy the crossing point assumption \eqref{e:omegabound}.  Then 
%\begin{equation}
%\begin{split}
%\frac{d}{dt}\left(f_e(\xi/2)-f_e(-\xi/2)\right) \leq& A\omega'(|\xi|)\left(\int\limits_0^{|\xi|} \frac{\omega(r)}{r}dr + |\xi| \int\limits_{|\xi|}^\infty \frac{\omega(r)}{r^2}dr\right) + A\omega(|\xi|)\int\limits_{|\xi|}^\infty \frac{\omega(|\xi|+r)-\omega(|\xi|)}{r^2}dr 
%\\&+\lambda \int\limits_0^{|\xi|} \frac{\delta_r \omega(|\xi|)+\delta_{-r} \omega(|\xi|)}{r^2}dr + \lambda\int\limits_{|\xi|}^\infty \frac{\omega(r+|\xi|)-\omega(r-|\xi|)-2\omega(|\xi|)}{r^2}dr
%\end{split}
%\end{equation}
%\end{lemma}

\begin{proof}

Let $G(\xi,h) = \delta_h f_e(\xi/2)K(\xi/2, h)-\delta_h f_e(-\xi/2)K(-\xi/2, h)$.  Then our goal is to bound 
\begin{equation}\label{e:trivialdiffusionsplit}
\int\limits_{\R^{d-1}} G(\xi, h) dh  = \int\limits_{|h|<|\xi|} G(\xi, h)dh + \int\limits_{|h|>|\xi|}G(\xi, h)dh,
\end{equation}
from above. We shall do so by bounded each of the two pieces on the right hand side of \eqref{e:trivialdiffusionsplit}.  

We shall start by bounding over the integral where $|h|>|\xi|$.  We first note that we can rewrite the sum defining $G$ in two different ways.  Namely, 
\begin{equation}
\begin{split}
G(\xi, h) &= (\delta_h f_e(\xi/2) - \delta_h f_e(-\xi/2))K(\xi/2, h) + \delta_{h}f_e(-\xi/2)(K(\xi/2, h)-K(-\xi/2, h))
\\&= (\delta_h f_e(\xi/2) - \delta_h f_e(-\xi/2))K(-\xi/2, h) + \delta_{h}f_e(\xi/2)(K(\xi/2, h)-K(-\xi/2, h))
\end{split}
\end{equation}
Recall by \eqref{e:omegabound} that  
\begin{equation}
\delta_h f_e(\xi/2) - \delta_h f_e(-\xi/2) = f_e(\xi/2+h)-f_e(-\xi/2+h)-\omega(|\xi|)\leq 0,
\end{equation}
for all $h\in \R^{2}$.  Hence as $K$ is uniformly elliptic, we get that
\begin{equation}\label{e:Gfirstbound}
\begin{split}
G(\xi, h) \leq& \lambda \frac{\delta_h f_e(\xi/2) - \delta_h f_e(-\xi/2)}{|h|^{3}} 
\\&+\delta_{h}f_e(\xi/2)_+(K(\xi/2, h)-K(-\xi/2, h))_+ + \delta_{h}f_e(-\xi/2)_{-}(K(\xi/2, h)-K(-\xi/2, h))_- 
\end{split}
\end{equation}
Again by \eqref{e:omegabound} we have that 
\begin{equation}\label{e:simpletouchbounds}
 \begin{split}
 \delta_hf(\xi/2)\leq \omega(|\xi+h|)-\omega(|\xi|)\leq \omega(|\xi|+|h|)-\omega(|\xi|),
 \\ \delta_hf(-\xi/2)\geq -\omega(|\xi-h|)+\omega(|\xi|)\geq -(\omega(|\xi|+|h|)-\omega(|\xi|)).
 \end{split}
 \end{equation}
Plugging \eqref{e:simpletouchbounds} into \eqref{e:Gfirstbound}, applying Lemma \ref{l:continuity}, and integrating over $|h|>|\xi|$ thus gives us that 
\begin{equation}
\begin{split}
\int\limits_{|h|>|\xi|}G(\xi, h)dh &\leq \int\limits_{|h|>|\xi|} \lambda \frac{\delta_h f_e(\xi/2) - \delta_h f_e(-\xi/2)}{|h|^{3}} + (\omega(|\xi|+|h|)-\omega(|\xi|))\ |K(\xi/2,h)-K(-\xi/2, h)| dh 
\\&\leq \int\limits_{|h|>|\xi|}\lambda \frac{\delta_h f_e(\xi/2) - \delta_h f_e(-\xi/2)}{|h|^{3}} + (\omega(|\xi|+|h|)-\omega(|\xi|))\frac{A\omega(|\xi|)}{|h|^3}dh, 
\\&\leq \lambda\int\limits_{|h|>|\xi|} \frac{\delta_h f_e(\xi/2) - \delta_h f_e(-\xi/2)}{|h|^{3}}dh +A\omega(|\xi|)\int\limits_{|\xi|}^\infty  \frac{\omega(|\xi|+r)-\omega(|\xi|)}{r^2}dr, 
\end{split}
\end{equation}
for some universal constant $A$ depending only on $||\nabla_x f_0||_{L^\infty}$.

Now we are left to bound the integral of $G$ for $|h|<|\xi|$.  We begin by adding and subtracting a linear term from $G$ to get
\begin{equation}
\begin{split}
G(\xi, h) &= (\delta_h f_e(\xi/2)-\omega'(|\xi|)\hat{\xi}\cdot h)K(\xi/2, h)-(\delta_h f_e(-\xi/2)-\omega'(\xi)\hat{\xi}\cdot h)K(-\xi/2, h) 
\\& \ + \omega'(\xi)\hat{\xi}\cdot h (K(\xi/2, h)-K(-\xi/2, h)).
\end{split}
\end{equation}

Similarly to \eqref{e:Gfirstbound} we can then bound 
\begin{equation}
\begin{split}
(\delta_h f_e(\xi/2)-\omega'(|\xi|)\hat{\xi}\cdot h)K(\xi/2, h)&-(\delta_h f_e(-\xi/2)-\omega'(\xi)\hat{\xi}\cdot h)K(-\xi/2, h)  \leq \lambda \frac{\delta_h f_e(\xi/2) - \delta_h f_e(-\xi/2)}{|h|^{3}} 
\\ &+(\delta_{h}f_e(\xi/2) - \omega'(|\xi|)\hat{\xi}\cdot h)_+(K(\xi/2, h)-K(-\xi/2, h))_+ 
\\&+ (\delta_{h}f_e(-\xi/2)-\omega'(|\xi|)\hat{\xi}\cdot h)_{-}(K(\xi/2, h)-K(-\xi/2, h))_- .
\end{split}
\end{equation}
In order to bound the error terms, we use that $\omega$ is a concave function of one variable to thus get
\begin{equation}
\omega(|h\pm \xi|)-\omega(|\xi|)\leq \omega'(|\xi|) \left( |\xi\pm h| - |\xi|\right) \leq \omega'(|\xi|)\hat{\xi}\cdot h +  \frac{\omega'(|\xi|)}{|\xi|} |h|^2.
\end{equation}
Hence, we have that 
\begin{equation}
\begin{split}
G(\xi,h) & \leq \lambda \frac{\delta_h f_e(\xi/2) - \delta_h f_e(-\xi/2)}{|h|^{3}} 
\\ &\  + \omega'(\xi)\hat{\xi}\cdot h (K(\xi/2, h)-K(-\xi/2, h))+\frac{\omega'(|\xi|)}{|\xi|} |h|^2 |K(\xi/2, h)-K(-\xi/2, h)|.
\end{split}
\end{equation}
Applying Lemmas \ref{l:asymmetry} and \ref{l:continuity} and integrating in space, we get 
\begin{equation}
\begin{split}
\int\limits_{|h|\leq |\xi|}G(\xi, h)dh - \lambda \int\limits_{|h|\leq |\xi|} \frac{\delta_h f_e(\xi/2) - \delta_h f_e(-\xi/2)}{|h|^{3}} 
&\lesssim  \omega'(|\xi|)\left(\int\limits_{|h|<|\xi|} \frac{\omega(|h|)}{|h|^{2}} + \frac{\omega(|\xi|)}{|\xi|} \frac{1}{|h|}dh\right)
\\&\lesssim \omega'(|\xi|)\left(\int\limits_0^{|\xi|} \frac{\omega(r)}{r}dr + \omega(|\xi|)\right)
\\& \lesssim \omega'(|\xi|)\int\limits_0^{|\xi|} \frac{\omega(r)}{r}dr,
\end{split}
\end{equation}
where the last inequality $\omega(|\xi|)\leq \displaystyle\int\limits_0^{|\xi|} \frac{\omega(r)}{r}dr$ is again due to the concavity of $\omega$.

\end{proof}

%%%%%%%%%%%%%%%%%%%%%%%%%%%%%%%%%%%%%%%%%%%%%%%%%%%%%%%%%%%%%%%%%%%%%

\section{Modulus Inequality}\label{s:modulusinequality}

\begin{lemma}\label{l:timederivbound}
Let $f:[0,T]\times \R^{2}\to \R$ be a smooth solution of the Muskat equation satisfy the crossing point assumption \eqref{e:omegabound} at some fixed time $t_0\in (0,T)$.  Assume the kernel $K$ defined in \eqref{e:Kdefn} satisfies the uniform ellipticity bounds \eqref{e:uniformellipticity}.  Then at the crossing point, 
\begin{equation}\label{e:timederivbound}
\begin{split}
\frac{d}{dt}\left(f_e(t,\xi/2)-f_e(t,-\xi/2)\right)\bigg|_{t=t_0} &\leq  A\omega'(|\xi|)\left(\int\limits_0^{|\xi|} \frac{\omega(r)}{r}dr + |\xi| \int\limits_{|\xi|}^\infty \frac{\omega(r)}{r^2}dr\right) + A\omega(|\xi|)\int\limits_{|\xi|}^\infty \frac{\omega(|\xi|+r)-\omega(|\xi|)}{r^2}dr 
\\&+c\lambda \int\limits_0^{|\xi|} \frac{\delta_r \omega(|\xi|)+\delta_{-r} \omega(|\xi|)}{r^2}dr + c\lambda\int\limits_{|\xi|}^\infty \frac{\omega(r+|\xi|)-\omega(r-|\xi|)-2\omega(|\xi|)}{r^2}dr,
\end{split}
\end{equation}
for some dimensional constant $c$ and constant $A$ depending only on $||\nabla_x f||_{L^\infty}$.  
\end{lemma}

To prove Lemma \ref{l:timederivbound}, we simply use the equation for $f_e$ \eqref{e:feequation}, our crossing point assumption \eqref{e:omegabound}, the estimates in lemmas \ref{l:asymmetry}, \ref{l:continuity}, \ref{l:diffusion} along side one final estimate due to Kiselev et al 
\begin{lemma}(See \cite{Kiselev} or \cite{KiselevIntRearrange} for more general case) \label{l:Kiselevrearrangement}
Let $f$ satisfy the crossing point assumptions \eqref{e:omegabound}.  Then 
\begin{equation}
\int\limits_{\R^{d-1}} \frac{\delta_h f_e(\xi/2)-\delta_h f_e(-\xi/2)}{|h|^d} dh \leq c_d \int\limits_0^{|\xi|} \frac{\delta_r \omega(|\xi|)+\delta_{-r} \omega(|\xi|)}{r^2}dr + c_d\int\limits_{|\xi|}^\infty \frac{\omega(r+|\xi|)-\omega(r-|\xi|)-2\omega(|\xi|)}{r^2}dr.
\end{equation}
\end{lemma}

\begin{proof} (Lemma \ref{l:timederivbound})

To begin, note by our crossing point assumption that 
\begin{equation}
\delta_h f_e(x) \leq \omega(|h|) \ \quad \forall x,h\in \R^{2}, \qquad f_e(\xi/2)-f_e(-\xi/2) = \omega(|\xi|).  
\end{equation}
Thus for any $h\in \R^{2}$, we get the bound 
\begin{equation}
f_e(\xi/2+h)-f_e(\xi/2) = f_e(\xi/2+h) - f_e(-\xi/2) - \omega(|\xi|) \leq \omega(|\xi+h|)-\omega(|\xi|),
\end{equation}
with equality at $h=0$.  Hence 
\begin{equation}
\nabla f_e(\xi/2) = \nabla \omega(|\xi|) = \omega'(|\xi|)\hat{\xi}.
\end{equation}
The same argument also tells us 
\begin{equation}
f_e(-\xi/2+h)-f_e(-\xi/2)\geq \omega(|\xi|)-\omega(|\xi-h|), \qquad \Rightarrow \qquad \nabla f_e(-\xi/2) = \omega'(|\xi|)\hat{\xi}.
\end{equation}
Using that $f_e$ solves the equation \eqref{e:feequation} and $\nabla_x f_e(\pm \xi/2) = \omega'(|\xi|)\hat{\xi}$, we thus get that 
\begin{equation}\label{e:timederivboundproof}
\begin{split}
\frac{d}{dt}\left(f_e(t,\xi/2)-f_e(t,-\xi/2)\right)\bigg|_{t=t_0} =&  \ \omega'(|\xi) \int\limits_{\R^{2}} \frac{h\cdot \hat{\xi}}{(\delta_hf(\xi/2)^2+|h|^2)^{3/2}} -\frac{h\cdot \hat{\xi}}{(\delta_hf(-\xi/2)^2+|h|^2)^{3/2}} dh
\\&+ \int\limits_{R^{2}} \delta_h f_e(\xi/2)K(\xi/2, h)-\delta_h f_e(-\xi/2)K(-\xi/2, h)dh.
\end{split}
\end{equation}
Rewriting the integral of the drift term as 
\begin{equation}
\begin{split}
\int\limits_{\R^{2}} \frac{h\cdot \hat{\xi}}{(\delta_hf(\xi/2)^2+|h|^2)^{3/2}} &-\frac{h\cdot \hat{\xi}}{(\delta_hf(-\xi/2)^2+|h|^2)^{3/2}} dh  = \int\limits_{|h|>|\xi|} \frac{h\cdot \hat{\xi}}{(\delta_hf(\xi/2)^2+|h|^2)^{3/2}} -\frac{h\cdot \hat{\xi}}{(\delta_hf(-\xi/2)^2+|h|^2)^{3/2}} dh 
\\&+\frac{1}{2}\int\limits_{|h|<|\xi|}  \frac{h\cdot \hat{\xi}}{(\delta_hf(\xi/2)^2+|h|^2)^{3/2}} - \frac{h\cdot \hat{\xi}}{(\delta_{-h}f(\xi/2)^2+|h|^2)^{3/2}} dh
\\&+\frac{1}{2}\int\limits_{|h|<|\xi|}\frac{h\cdot \hat{\xi}}{(\delta_{-h}f(-\xi/2)^2+|h|^2)^{3/2}}  -\frac{h\cdot \hat{\xi}}{(\delta_hf(-\xi/2)^2+|h|^2)^{3/2}} dh,
\end{split}
\end{equation}
and applying lemmas \ref{l:asymmetry} and \ref{l:continuity}, we see that 
\begin{equation}\label{e:driftbound}
 \omega'(|\xi) \int\limits_{\R^{2}} \frac{h\cdot \hat{\xi}}{(\delta_hf(\xi/2)^2+|h|^2)^{3/2}} -\frac{h\cdot \hat{\xi}}{(\delta_hf(-\xi/2)^2+|h|^2)^{3/2}} dh\lesssim \omega'(|\xi|)\left(\int\limits_0^{|\xi|}\frac{\omega(r)}{r}dr + |\xi|\int\limits_{|\xi|}^\infty \frac{\omega(r)}{r^2}dr\right).
\end{equation}
By Lemma \ref{l:diffusion}, we can bound the difference in diffusions as 
\begin{equation}\label{e:diffusionbound}
\begin{split}
\int\limits_{R^{2}} \delta_h f_e(\xi/2)K(\xi/2, h)-\delta_h f_e(-\xi/2)K(-\xi/2, h)dh &- \lambda \int\limits_{\R^{2}}\frac{\delta_h f_e(\xi/2)-\delta_h f_e(-\xi/2)}{|h|^3} dh 
\\&\lesssim  \omega'(|\xi|)\int\limits_{0}^{|\xi|} \frac{\omega(r)}{r}dr + \omega(|\xi|)\int\limits_{|\xi|}^\infty \frac{\omega(|\xi|+r)-\omega(|\xi|)}{r^2}dr.
\end{split}
\end{equation}
Finally, applying Lemma \ref{l:Kiselevrearrangement} and plugging \eqref{e:driftbound}, \eqref{e:diffusionbound} into \eqref{e:timederivboundproof} gives us \eqref{e:timederivbound}
\end{proof}

%From Kiselev paper, we have that 
%\begin{equation}\label{e:modulusbound}
%\int\limits_{\R^{d-1}} \frac{\delta_h f_e(\xi/2)-\delta_h f_e(-\xi/2)}{|h|^d} dh \leq c\int\limits_0^{|\xi|} \frac{\delta_r \omega(|\xi|)+\delta_{-r} \omega(|\xi|)}{r^2}dr + c\int\limits_{|\xi|}^\infty \frac{\omega(r+|\xi|)-\omega(r-|\xi|)-2\omega(|\xi|)}{r^2}dr.
%\end{equation}
%for some dimensional constant $c$.  

%Abusing notation by replacing $\lambda$ with $c\lambda$, we thus have that 
%\begin{equation}
%\begin{split}
%\frac{d}{dt}\left(f_e(\xi/2)-f_e(-\xi/2)\right) \leq& A\omega'(|\xi|)\left(\int\limits_0^{|\xi|} \frac{\omega(r)}{r}dr + |\xi| \int\limits_{|\xi|}^\infty \frac{\omega(r)}{r^2}dr\right) + A\omega(|\xi|)\int\limits_{|\xi|}^\infty \frac{\omega(|\xi|+r)-\omega(|\xi|)}{r^2}dr 
%\\&+\lambda \int\limits_0^{|\xi|} \frac{\delta_r \omega(|\xi|)+\delta_{-r} \omega(|\xi|)}{r^2}dr + \lambda\int\limits_{|\xi|}^\infty \frac{\omega(r+|\xi|)-\omega(r-|\xi|)-2\omega(|\xi|)}{r^2}dr
%\end{split}
%\end{equation}

We are now nearly ready to complete the breakthrough argument of section \ref{s:breakthrough}.  If our goal was to prove propagation of a modulus of continuity $\omega$ rather than the generation of one $\rho$, it would suffice to construct some function $\omega$ such that 
\begin{equation}
\begin{split}
A\omega'(|\xi|)\left(\int\limits_0^{|\xi|} \frac{\omega(r)}{r}dr + |\xi| \int\limits_{|\xi|}^\infty \frac{\omega(r)}{r^2}dr\right) + A\omega(|\xi|)\int\limits_{|\xi|}^\infty \frac{\omega(|\xi|+r)-\omega(|\xi|)}{r^2}dr 
\\+c\lambda \int\limits_0^{|\xi|} \frac{\delta_r \omega(|\xi|)+\delta_{-r} \omega(|\xi|)}{r^2}dr + c\lambda\int\limits_{|\xi|}^\infty \frac{\omega(r+|\xi|)-\omega(r-|\xi|)-2\omega(|\xi|)}{r^2}dr<0,
\end{split}
\end{equation}
which would give the equivalent contradiction to \eqref{e:finalinequality}.  

As our goal is generation of a modulus of continuity though, we will need to prove the (marginally) stronger inequality 
\begin{equation}
\begin{split}
A\omega'(|\xi|)\left(\int\limits_0^{|\xi|} \frac{\omega(r)}{r}dr + |\xi| \int\limits_{|\xi|}^\infty \frac{\omega(r)}{r^2}dr\right) + A\omega(|\xi|)\int\limits_{|\xi|}^\infty \frac{\omega(|\xi|+r)-\omega(|\xi|)}{r^2}dr 
\\+c\lambda \int\limits_0^{|\xi|} \frac{\delta_r \omega(|\xi|)+\delta_{-r} \omega(|\xi|)}{r^2}dr + c\lambda\int\limits_{|\xi|}^\infty \frac{\omega(r+|\xi|)-\omega(r-|\xi|)-2\omega(|\xi|)}{r^2}dr<-\omega'(|\xi|)\omega(|\xi|).
\end{split}
\end{equation}

Luckily, in \cite{Kiselev}, the authors were able to prove that 
\begin{lemma}\label{l:Kiselev}\cite{Kiselev}

Let $\omega:[0,\infty)\to [0,\infty)$ be the modulus of continuity defined by 
\begin{equation}\label{e:omegadefn}
\left\{\begin{array}{cl} \omega(\xi) = \xi-\xi^{3/2}, & 0\leq \xi \leq \delta \\ \omega'(\xi) = \displaystyle\frac{\gamma}{\xi(4+\log(\xi/\delta))}, & \xi\geq \delta \end{array}\right. .
\end{equation}
Then this modulus satisfies
\begin{equation}\label{e:Kiselevcalc}
\begin{split}
A\omega'(|\xi|)&\left(\int\limits_0^{|\xi|} \frac{\omega(r)}{r}dr + |\xi|\int\limits_{|\xi|}^\infty \frac{\omega(r)}{r^2} dr \right) 
\\&+\frac{c\lambda}{2}\int\limits_0^{|\xi| }\frac{\delta_r\omega(|\xi|) + \delta_{-r}\omega(|\xi|)}{r^2}dr + \frac{c\lambda}{2}\int\limits_{|\xi|}^\infty \frac{\omega(r+|\xi|)-\omega(r-|\xi|)-2\omega(|\xi|)}{r^2}dr < 0,
\end{split}
\end{equation}
for all $|\xi|>0$, so long as $\delta$ is taken sufficiently small depending on $\displaystyle\frac{A}{c\lambda}$, and $\gamma $ is sufficiently small depending on $\delta, \displaystyle\frac{A}{c\lambda}$.
\end{lemma}

Following the proof of Lemma \ref{l:Kiselev}, we can similarly show that the same modulus of continuity $\omega$ satisfies the intero-differential inequality we need.  
\begin{lemma}\label{l:omegacalc}
Let $\omega$ be as in \eqref{e:omegadefn}.  Then this modulus satisfies 
\begin{equation}
\begin{split}
A\omega'(|\xi|)\left(\int\limits_0^{|\xi|} \frac{\omega(r)}{r}dr + |\xi| \int\limits_{|\xi|}^\infty \frac{\omega(r)}{r^2}dr\right) + A\omega(|\xi|)\int\limits_{|\xi|}^\infty \frac{\omega(|\xi|+r)-\omega(|\xi|)}{r^2}dr 
\\+c\lambda \int\limits_0^{|\xi|} \frac{\delta_r \omega(|\xi|)+\delta_{-r} \omega(|\xi|)}{r^2}dr + c\lambda\int\limits_{|\xi|}^\infty \frac{\omega(r+|\xi|)-\omega(r-|\xi|)-2\omega(|\xi|)}{r^2}dr<-\omega'(|\xi|)\omega(|\xi|),
\end{split}
\end{equation}
so long as $\delta$ is taken sufficiently small depending on $\displaystyle\frac{A+1}{c\lambda}$, and $\gamma $ is sufficiently small depending on $\delta, \displaystyle\frac{A+1}{c\lambda}$.
\end{lemma}

\begin{proof}
To begin, note that $\omega$ is concave so long as $\gamma $ is taken sufficiently small depending on $\delta$. Hence, $$\omega(|\xi|)\leq \displaystyle\int\limits_{0}^{|\xi|} \frac{\omega(r)}{r}dr.$$   
Abusing notation and replacing $A+1$ by $A$ and $\displaystyle\frac{c\lambda}{2}$ by $\lambda$, in light of Lemma \ref{l:Kiselev} it suffices to prove that 
\begin{equation}\label{e:inequalitytoprove}
A\omega(|\xi|)\int\limits_{|\xi|}^\infty \frac{\omega(|\xi|+r)-\omega(|\xi|)}{r^2}dr 
+\lambda \int\limits_0^{|\xi|} \frac{\delta_r \omega(|\xi|)+\delta_{-r} \omega(|\xi|)}{r^2}dr + \lambda\int\limits_{|\xi|}^\infty \frac{\omega(r+|\xi|)-\omega(r-|\xi|)-2\omega(|\xi|)}{r^2}dr <0.
\end{equation}

Note that as $\omega$ is concave, the latter two integrals of \eqref{e:inequalitytoprove} are necessarily nonpostiive.  Depending on the size of $|\xi|$, we shall rely on one or the other to control the error term.  
That leaves us with two cases to check.  

{\bf Case 1: } $|\xi| \leq \delta$

We have that in this case, 
\begin{equation}
\int\limits_0^{|\xi|} \frac{\delta_r \omega(|\xi|)+\delta_{-r} \omega(|\xi|)}{r^2}dr  \leq |\xi|\omega''(|\xi|) = -\frac{3}{2}\xi\xi^{-1/2}.
\end{equation}

We also have the bounds
\begin{equation}
\begin{split}
\int\limits_{|\xi|}^\delta \frac{\omega(|\xi|+r)-\omega(|\xi|)}{r^2}dr &\leq \int\limits_{|\xi|}^\delta \frac{\omega(r)}{r^2}dr\leq  \int\limits_{|\xi|}^\delta \frac{1}{r}dr = \log(\delta/|\xi|),
\\ \int\limits_{\delta}^\infty  \frac{\omega(|\xi|+r)-\omega(|\xi|)}{r^2}dr &\leq \int\limits_{\delta}^\infty \frac{\omega(r)}{r^2}dr\leq \frac{\omega(\delta)}{\delta}+\gamma\int\limits_{\delta}^\infty \frac{1}{r^2 (4+\log(r/\delta))}dr \leq 1+\frac{\gamma}{4\delta}<2,
\end{split}
\end{equation}
where the last inequality follows by taking $\gamma<4\delta$.  

Putting this together, we thus have that 

\begin{equation}
\begin{split}
A\omega(|\xi|)\int\limits_{|\xi|}^\infty \frac{\omega(|\xi|+r)-\omega(|\xi|)}{r^2}dr +\lambda \int\limits_0^{|\xi|} \frac{\delta_r \omega(|\xi|)+\delta_{-r} \omega(|\xi|)}{r^2}dr &\leq |\xi| (A(2+\log(\delta/|\xi|))-\frac{3}{2}\lambda |\xi|^{-1/2}) 
\\&\leq -|\xi| <0
\end{split}
\end{equation}
so long as $\delta$ is taken sufficiently small.

{\bf Case 2: } $|\xi| \geq \delta$

To begin, note that $\omega(r+|\xi|)-\omega(r-|\xi|)\leq \omega(2|\xi|)$ since $\omega$ is concave.  Then we can also bound
\begin{equation}
\omega(2|\xi|)\leq \omega(|\xi|)+\int\limits_{|\xi|}^{2|\xi|}\displaystyle\frac{\gamma}{\xi(4+\log(\xi/\delta))} \leq \omega(|\xi|)+\frac{\log(2)\gamma}{4} \leq \omega(|\xi|)+\frac{\omega(\delta)}{2}\leq \frac{3}{2}\omega(|\xi|),
\end{equation}
so long as $\gamma$ is taken sufficiently small depending on $\delta$.  Hence, 
\begin{equation}
\int\limits_{|\xi|}^\infty \frac{\omega(r+|\xi|)-\omega(r-|\xi|)-2\omega(|\xi|)}{r^2}dr \leq\int\limits_{|\xi|}^\infty \frac{-\omega(|\xi|)}{2r^2}dr=\frac{-1}{2}\frac{\omega(|\xi|)}{|\xi|}
\end{equation}

Using the same argument, we can also bound
\begin{equation}
\int\limits_{|\xi|}^\infty \frac{\omega(|\xi|+r)-\omega(|\xi|)}{r^2}dr \leq \frac{\omega(2|\xi|)-\omega(|\xi|)}{|\xi|} + \int\limits_{|\xi|}^\infty \frac{\gamma}{r^2}dr \leq  \frac{\omega(\delta)}{2|\xi|} +\frac{\gamma}{|\xi|}\leq \frac{\lambda}{4A|\xi|}
\end{equation}
so long as $\delta, \gamma$ are taken sufficiently small.  Hence, 
\begin{equation}
\begin{split}
A\omega(|\xi|)\int\limits_{|\xi|}^\infty \frac{\omega(|\xi|+r)-\omega(|\xi|)}{r^2}dr +\lambda \int\limits_{|\xi|}^\infty \frac{\omega(r+|\xi|)-\omega(r-|\xi|)-2\omega(|\xi|)}{r^2}dr &\leq-\frac{\lambda}{4}\frac{\omega(|\xi|)}{|\xi|} <0.
\end{split}
\end{equation}

\end{proof}

%%%%%%%%%%%%%%%%%%%%%%%%%%%%%%%%%%%%%%%%%%%%%%%%%%%%%%%%%%%%%%%%%%%%%%%

\section{Our choice for the modulus $\overline{\omega}$}\label{s:modulusrho}
We've now shown that for the modulus defined in \eqref{e:omegadefn} that if the assumptions \eqref{e:omegabound} hold that
\begin{equation}\label{e:omegafinalequation}
\frac{d}{dt}\left(f_e(t,\xi/2)-f_e(t,-\xi/2)\right)\bigg|_{t=T} < -\omega'(|\xi|)\omega(|\xi|).
\end{equation}
We claim that in fact \eqref{e:omegafinalequation} will hold for any rescaling $\omega_R(|\xi|) = \omega(R|\xi|)$ as well.  

\begin{lemma}\label{l:rescaling}
Let $R>0$, and $\omega_R(|\xi|)= \omega(R|\xi|)$, where $\omega$ is such that Lemma \ref{l:omegacalc} holds.  Then Lemma \ref{l:omegacalc} holds for $\omega_R$ as well.  That is, for any $|\xi|>0$,
\begin{equation}\label{e:rescaledomegainequal}
\begin{split}
A\omega_R'(|\xi|)\left(\int\limits_0^{|\xi|} \frac{\omega_R(r)}{r}dr + |\xi| \int\limits_{|\xi|}^\infty \frac{\omega_R(r)}{r^2}dr\right) + A\omega_R(|\xi|)\int\limits_{|\xi|}^\infty \frac{\omega_R(|\xi|+r)-\omega_R(|\xi|)}{r^2}dr 
\\+c_d\lambda \int\limits_0^{|\xi|} \frac{\delta_r \omega_R(|\xi|)+\delta_{-r} \omega_R(|\xi|)}{r^2}dr + c_d\lambda\int\limits_{|\xi|}^\infty \frac{\omega_R(r+|\xi|)-\omega_R(r-|\xi|)-2\omega_R(|\xi|)}{r^2}dr<-\omega_R'(|\xi|)\omega_R(|\xi|).
\end{split}
\end{equation}

\end{lemma}
\begin{proof}
Given $g\in \dot{W}^{1,\infty}([0,\infty); [0,\infty))$, we define the function $F[g]: (0,\infty)\to \R$ by 
\begin{equation}
\begin{split}
F[g](|\xi|)= 
Ag'(|\xi|)\left(\int\limits_0^{|\xi|} \frac{g(r)}{r}dr + |\xi| \int\limits_{|\xi|}^\infty \frac{g(r)}{r^2}dr\right) + Ag(|\xi|)\int\limits_{|\xi|}^\infty \frac{g(|\xi|+r)-g(|\xi|)}{r^2}dr 
\\+c_d\lambda \int\limits_0^{|\xi|} \frac{\delta_r g(|\xi|)+\delta_{-r} g(|\xi|)}{r^2}dr + c_d\lambda\int\limits_{|\xi|}^\infty \frac{g(r+|\xi|)-g(r-|\xi|)-2g(|\xi|)}{r^2}dr+g'(|\xi|)g(|\xi|).
\end{split}
\end{equation}
Letting $g_R(|\xi|) = g(R|\xi|)$, direct calculation gives that 
\begin{equation}
F[g_R](|\xi|) = RF[g](R|\xi|).  
\end{equation}

Hence, 
\begin{equation}
F[\omega](|\xi|)<0 \qquad \forall |\xi|>0, \qquad \iff \qquad F[\omega_R](|\xi|)<0, \qquad \forall |\xi|>0.
\end{equation}

\end{proof}

\begin{remark}
We note that Lemma \ref{l:rescaling} is a natural consequence of the fact that solutions to the Muskat equation \eqref{e:fequation} are preserved under the geometric rescaling $R^{-1}f(Rt,Rx)$ and the fact that the constants $A, \lambda$ appearing in \eqref{e:rescaledomegainequal} depend only on the scale invariant quantity $||\nabla f_0||_{L^\infty}$ and dimension.  
\end{remark}

\begin{lemma}\label{l:generationrho}
Let $f_0 \in \dot{W}^{1,\infty}(\R^{2})$ with $||\nabla f_0||_{L^\infty} < \displaystyle\frac{1}{\sqrt{5}}$, and $f:[0,T]\times \R^{2}\to \R$ be a sufficiently smooth solution to the Muskat equation satisfying the assumptions of the breakthrough argument.  Letting $\omega$ be as defined in \eqref{e:omegadefn} satisfying Lemma \ref{l:omegacalc}, and taking 
\begin{equation}\label{e:Cequation}
\begin{split}
C =  \sup\limits_{0<r<\omega^{-1}(2||\nabla f_0||_{L^\infty})} \frac{r}{\omega(r)} = \frac{\omega^{-1}(2||\nabla f_0||_{L^\infty})}{2||\nabla f_0||_{L^\infty}}, 
\\ \overline{\omega}(r) = \omega(Cr).
\end{split}
\end{equation}
If $\nabla_x f_0$ has modulus $\overline{\omega}\left(\frac{\cdot}{\delta}\right)$, then $\nabla_x f(t,\cdot)$ has modulus $\overline{\omega}\left(\frac{\cdot}{t+\delta}\right)$.  In particular taking $\delta=0$,  we for all such solutions $f$ that  
\begin{equation}
|\nabla_x f(t,x)-\nabla_x f(t,y)|< \overline{\omega}\left(\frac{|x-y|}{t}\right).
\end{equation}
\end{lemma}

\begin{proof}
We focus on the case that $\delta = 0$.  The case where $\delta>0$ follows by a simple modification of the breakthrough argument given in Section \ref{s:breakthrough}.  

The constant $C$ defined in \eqref{e:Cequation} was chosen so that we have the inequality
\begin{equation}
\overline{\omega}(r)\geq r, \qquad \forall r \ \mbox{s.t.} \ 0\leq \overline{\omega}(r)\leq 2||\nabla_x f_0||_{L^\infty}.  
\end{equation}

Recall the breakthrough argument of section \ref{s:breakthrough}.  We have that if $\nabla_x f$ does not have the modulus $\rho\left(\displaystyle\frac{\cdot}{t}\right)$ for all times $t$, then necessarily we can find a positive time $t_0$ and points $x,y\in \R^{2}$ and direction $e\in S^{1}$ such that 
\begin{equation}
\rho\left(\frac{|x-y|}{t_0}\right) = f_e(t_0,x)-f_e(t_0,y), \quad \Rightarrow -\displaystyle\frac{|x-y|}{t_0^2}\overline{\omega}'\left(\frac{|x-y|}{t_0}\right) \leq \frac{d}{dt}\left(f_e(t,x)-f_e(t,y)\right)\bigg|_{t=t_0}.
\end{equation}

By the maximum principle for the slope (see Proposition \ref{p:maximumprinciple}), we thus have that at a point of equality
\begin{equation}
\begin{split}
\overline{\omega}\left(\frac{|x-y|}{t_0}\right) = f_e(t_0,x)-f_e(t_0,y)\leq 2||\nabla f_0||_{L^\infty},
\\ \Rightarrow \quad \overline{\omega}\left(\frac{|x-y|}{t_0}\right) \geq \frac{|x-y|}{t_0}.
\end{split}
\end{equation}
Hence, applying Lemmas \ref{l:timederivbound} and \ref{l:rescaling}, we have that 
\begin{equation}
\frac{d}{dt}\left(f_e(t,x)-f_e(t,y)\right)\bigg|_{t=t_0} < - \frac{1}{t_0}\overline{\omega}'\left(\frac{|x-y|}{t_0}\right) \overline{\omega}\left(\frac{|x-y|}{t_0}\right) \leq \frac{-|x-y|}{t_0^2}\overline{\omega}'\left(\frac{|x-y|}{t_0}\right),
\end{equation}
a contradiction.  Thus $\nabla_x f(t,\cdot)$ must have the modulus $\overline{\omega}\left(\frac{\cdot}{t}\right) $ for all times $t$.  

\end{proof}

\begin{corollary}\label{c:C11bound}
Let $f:[0,T]\times \R^2\to \R$ be as in Lemma \ref{l:generationrho}.  Then 
\begin{equation}
||D^2_x f||_{L^\infty}(t) \lesssim \min\left\{||D^2_x f_0||_{L^\infty}, \frac{1}{t}\right\}.
\end{equation}
\end{corollary}

%%%%%%%%%%%%%%%%%%%%%%%%%%%%%%%%%%%%%%%%%%%%%%%%%%%%%%%%%%%%%%%%%%%%%%%

\section{Comparison Principle and Uniqueness}\label{s:comparison}

Our goal in this section is to show that under suitable bounds on the slope, classical solutions to the Muskat equation obey the comparison principle.  In order to make our exact assumptions and claims clear, we start off with a definition of what precisely we mean by a classical solution.  As the comparison principle has not been proven before in any dimension,  we perform all calculations here for a general dimension $d$.  

\begin{definition}\label{d:classical}(Classical solution)
We say that a function $f: [0,T]\times \R^{d-1} \to \R$ is a classical solution to the $d$-dimensional Muskat equation \eqref{e:dequation} if 
\begin{enumerate}

\item $f\in C^1((0,T)\times \R^{d-1})$, and with uniform limits at initial and final times 
\begin{equation}
\lim\limits_{\epsilon\to 0+} ||f(\epsilon,\cdot)-f(0,\cdot)||_{L^\infty(\R^{d-1})}, ||f(T, \cdot)-f(T-\epsilon,\cdot)||_{L^\infty(\R^{d-1})} = 0.
\end{equation}
\item For any compact subset $\mathcal{K}\subseteq (0,T)\times \R^{d-1}$, there exist smooth, nonnegative functions 
$\omega, \Omega: [0,\infty)\to [0,\infty)$ such that 
\begin{equation}
\begin{split}
\sup\limits_{(t,x)\in\mathcal{K}} \sup\limits_{h\in B_r} |f(t, x+h)-f(t,x)-\nabla_x f(t,x)\cdot h |\leq \omega(r)r,
\\\sup\limits_{(t,x)\in\mathcal{K}}  \sup\limits_{h\in B_R} |f(t,x+h)-f(t,x)|\leq \Omega(R),
\\ \int\limits_{0}^1 \frac{\omega(r)}{r}dr + \int\limits_{1}^\infty \frac{\Omega(R)}{R^2}dR < \infty.
\end{split}
\end{equation}
\item For all $(t,x)\in (0,T)\times \R^{d-1}$, the function $f$ satisfies the integro-differential equality 
\begin{equation}\label{e:dequation}
\partial_t f(t,x)= P.V.\int\limits_{\R^{d-1}} \frac{f(t,y)-f(t,x) - \nabla_x f(t,x)\cdot (y-x)}{((f(t,y)-f(t,x))^2 + |y-x|^2)^{d/2}} dy.
\end{equation}
\end{enumerate}
A function $f$ is a classical sub or super solution if in assumption 3. we replace the equality by the inequalities $\leq, \geq$ respectively.
\end{definition}

We note that assumptions 1. and  2. guarantees that the integral in assumption 3. is well defined in the principle value sense for any $(t,x)\in (0,T)\times \R^{d-1}$.

\begin{lemma}\label{l:comparison}
Let $|a|\leq \displaystyle\frac{1}{\sqrt{2d-1}}$.  Then the function $t \to \displaystyle\frac{t- a}{(t^2+1)^{d/2}}$ is monotonic increasing for $|t|\leq \displaystyle\frac{1}{\sqrt{2d-1}}$.  
\end{lemma}

\begin{proof}
Differentiating the function, we see that 
$$\displaystyle\frac{d}{dt}  \displaystyle\frac{t- a}{(t^2+1)^{d/2}} \geq 0 \quad \iff \quad (t^2+1) -dt(t-a) = (1-d)t^2 + adt + 1\geq 0.$$
By our assumption on $a$, we have that this is true whenever $|t|\leq \displaystyle\frac{1}{\sqrt{2d-1}}$.
\end{proof}

\begin{theorem}\label{t:comparison2}(Comparison Principle) 
Let $f,g: [0,T]\times \R^{d-1}\to \R$ be classical sub/super solutions of the $d$-dimensional Muskat equation \eqref{e:dequation} for $d=2$ or 3.  Assume that $|| \nabla f||_{L^\infty} , ||\nabla g ||_{L^\infty} < \displaystyle\frac{1}{\sqrt{2d-1}}$ and $f(0,x)\leq g(0,x)$.  Then $f(t,x)\leq g(t,x)$ for all $t\geq 0$.  
\end{theorem}

\begin{proof}

By replacing $g(t,x)$ with $g(t,x) + \eta +\mu t$ for some $\eta,\mu>0$, we can assume without loss of generality that 

\begin{equation}
\begin{split}
f(0,x)+\eta \leq g(0,x), 
\\ \delta_t g(t,x) \geq \mu + \int\limits_{\R^{d-1}} \frac{\delta_h g(t,x) - \nabla g(t,x)\cdot h}{(\delta_h g(t,x)^2+|h|^2)^{d/2}}dh.
\end{split}
\end{equation}

Now let $\phi: \R^{d-1} \to [0,\infty)$ be such that $\nabla \phi\in W^{1,\infty}(\R^{d-1}; \R^{d-1})$,  with 
\begin{equation}
\sup\limits_{x} \int\limits_{|h|>1} \frac{|\delta_h \phi(x)|}{|h|^{d}}dh \leq C.
\end{equation}
and 
\begin{equation}\label{e:phifinequal}
\limsup\limits_{x\to \infty} \sup\limits_{t\in [0,T]}\frac{f(t,x)-g(t,x)}{ \phi(x)} \leq 0.
\end{equation}
Such a function $\phi$ necessarily exists by the definition of a classical sub/ super solution.  

Let $\epsilon>0$, and $g_\epsilon (t,x)= g(t,x)+\epsilon \phi(x)$.  It suffices to show that 
\begin{equation}
f(t,x)\leq g_\epsilon(t,x),\qquad (t,x)\in [0,T]\times \R^{d-1}.
\end{equation}

By \eqref{e:phifinequal} it follows that there is an $R_\epsilon > 0$ such that 
\begin{equation}
f(t,x)< g_\epsilon (t,x), \qquad t\in [0,T], |x|> R_\epsilon.
\end{equation}
Similarly as $f(0,x)+\eta \leq g_\epsilon (0,x)$, it follows by continuity of $f,g$ that there is a $\delta>0$ such that
\begin{equation}
f(t,x)<g_\epsilon (t,x), \qquad t<\delta, x\in \overline{B_{R_\epsilon}}.
\end{equation}

Thus if $f(t,x)>g_\epsilon(t,x)$ at some point $(t,x)$, then there must exist a first crossing point point.  That is, there must be a point $(t_0,x_0)\in (0,T]\times\overline{B_{R_\epsilon}}$ such that 

\begin{equation}
\begin{split}
f(t,x)\leq g_\epsilon (t,x)&, \quad \forall (t,x)\in [0,t_0]\times \R^{d-1}, 
\\ f(t_0,x_0)=g_\epsilon(t_0,x_0).
\end{split}
\end{equation}
At this crossing point, we have that 
\begin{equation}\label{e:crossingpointbounds}
\begin{split}
\partial_t f(t_0,x_0)\geq \partial_t g_\epsilon(t_0,x_0),
\\ \nabla f(t_0,x_0)=\nabla g_\epsilon (t_0,x_0),
\\ \delta_hf(t_0,x_0) \leq \delta_h g_\epsilon(t_0,x_0).
\end{split}
\end{equation}

Let $\epsilon <<1$.  Then we claim that $g_\epsilon $ is a strict superoslution of Muskat on $[\delta, T]\times \R^{d-1}$ with $||\nabla g_\epsilon ||_{L^\infty} < \displaystyle\frac{1}{\sqrt{2d-1}}$.  It then follows by the Muskat equation and Lemma \ref{l:comparison} that 
\begin{equation}
\begin{split}
\partial_t f(t_0,x_0) &\leq \int\limits_{\R^{d-1}} \frac{\delta_h f(t_0,x_0) - \nabla f(t_0,x_0)\cdot h}{(\delta_h f(t_0,x_0)^2+|h|^2)^{d/2}} dh 
\\&\leq \int\limits_{\R^{d-1}} \frac{\delta_h g_\epsilon(t_0,x_0) - \nabla g_\epsilon(t_0,x_0)\cdot h}{(\delta_h g_\epsilon(t_0,x_0)^2+|h|^2)^{d/2}} dh
\\&<\partial_t g_\epsilon (t_0,x_0),
\end{split}
\end{equation}
contradicting \eqref{e:crossingpointbounds}.  

Since $||\nabla g||_{L^\infty} < \displaystyle\frac{1}{\sqrt{2d-1}}$ and $\phi \in \dot{W}^{1,\infty}$ is smooth, its clear that for $\epsilon$ sufficiently small that $||\nabla g_\epsilon ||_{L^\infty} < \displaystyle\frac{1}{\sqrt{2d-1}}$.  So now we just need to show that $g_\epsilon$ is a strict super solution.  Direct calculation gives
\begin{equation}
\begin{split}\label{e:gepsilonsuper}
\int\limits_{\R^{d-1}} \frac{\delta_h g_\epsilon(t,x) - \nabla g_\epsilon(t,x)\cdot h}{(\delta_h g_\epsilon(t,x)^2+|h|^2)^{d/2}} dh&= \int\limits_{\R^{d-1}} \frac{\delta_h g(t,x) - \nabla g(t,x)\cdot h}{(\delta_h g(t,x)^2+|h|^2)^{d/2}} dh+\epsilon \int\limits_{\R^{d-1}} \frac{\delta_h \phi(x) - \nabla \phi(x)\cdot h}{(\delta_h g_\epsilon(t,x)^2+|h|^2)^{d/2}} dh
\\&\qquad +\int\limits_{\R^{d-1}}\left(\frac{\delta_h g(t,x)-\nabla g(t,x)\cdot h}{(\delta_h g_\epsilon(t,x)^2+|h|^2)^{d/2}} -\frac{\delta_h g(t,x)-\nabla g(t,x)\cdot h}{(\delta_h g(t,x)^2+|h|^2)^{d/2}}\right)dh
\\&\leq \partial_t g_\epsilon (t,x)-\mu +\epsilon \int\limits_{\R^{d-1}} \frac{\delta_h \phi(x) - \nabla \phi(x)\cdot h}{(\delta_h g_\epsilon(t,x)^2+|h|^2)^{d/2}} dh
\\&\qquad +\int\limits_{\R^{d-1}}\left(\frac{\delta_h g(t,x)-\nabla g(t,x)\cdot h}{(\delta_h g_\epsilon(t,x)^2+|h|^2)^{d/2}} -\frac{\delta_h g(t,x)-\nabla g(t,x)\cdot h}{(\delta_h g(t,x)^2+|h|^2)^{d/2}}\right)dh
\end{split}
\end{equation}

As $\phi$ is smooth and $g$ is a classical supersolution, for any $(t,x)\in [\delta, T]\times \R^{d-1}$ we can bound
\begin{equation}
\begin{split}
 \bigg|\int\limits_{\R^{d-1}} \frac{\delta_h \phi(x_0) - \nabla \phi(x_0)\cdot h}{(\delta_h g_\epsilon(t,x)^2+|h|^2)^{d/2}} dh\bigg| 
 &\leq||D^2\phi||_{L^\infty}\int\limits_{|h|\leq 1} \frac{1}{|h|^{d-2}}dh + \int\limits_{|h|>1} \frac{|\delta_h \phi(x)|}{|h|^d}dh 
 \\& \quad + \frac{||\nabla \phi||_{L^\infty}}{2}\int\limits_{|h|>1} \bigg| \frac{|h|}{(\delta_{h} g_\epsilon(t,x)^2+|h|^2)^{d/2}} - \frac{|h|}{(\delta_{-h} g_\epsilon(t,x)^2+|h|^2)^{d/2}} dh\bigg|
 \\&\leq C(\phi,g,\delta)
 \end{split}
\end{equation}
Hence taking $\epsilon < \frac{\mu}{3C(\phi,g,\delta)}$ we have that 
\begin{equation}\label{e:gepsilonsuper1}
\epsilon \int\limits_{\R^{d-1}} \frac{\delta_h \phi(x) - \nabla \phi(x)\cdot h}{(\delta_h g_\epsilon(t,x)^2+|h|^2)^{d/2}} dh <\frac{\mu}{3}.
\end{equation}

Similarly, since $g$ is a classical supersolution there is a $r = r(\mu,\delta)>0$ such that for any $(t,x)\in [\delta, T]\times \R^d$  
\begin{equation}\label{e:gepsilonsuper2}
\int\limits_{|h|<r} \frac{|\delta_h g(t,x)-\nabla g(t,x)\cdot h|}{|h|^{d}}dh < \frac{\mu}{3}.
\end{equation}
As we have that 
\begin{equation}
\bigg|\frac{1}{(\delta_h g_\epsilon(t,x)^2+|h|^2)^{d/2}} -\frac{1}{(\delta_h g(t,x)^2+|h|^2)^{d/2}}\bigg| \leq\epsilon \frac{ C(d)|\delta_h \phi(x)|}{|h|^{d+1}}
\end{equation}
we can similarly bound
\begin{equation}\label{e:gepsilonsuper3}
\begin{split}
\int\limits_{|h|>r}\left(\frac{\delta_h g(t,x)-\nabla g(t,x)\cdot h}{(\delta_h g_\epsilon(t,x)^2+|h|^2)^{d/2}} -\frac{\delta_h g(t,x)-\nabla g(t,x)\cdot h}{(\delta_h g(t,x)^2+|h|^2)^{d/2}}\right)dh 
&\leq \epsilon \int\limits_{|h|>r} |\delta_h g(t,x)-\nabla g(t,x)\cdot h| \frac{ C(d)|\delta_h \phi(x)|}{|h|^{d+1}} dh 
\\& \leq 2\epsilon C(d) ||\nabla g||_{L^\infty} \int\limits_{|h|>r}\frac{|\delta_h \phi(x)|}{|h|^d}dh 
\\& \leq \frac{\mu}{3},
\end{split}
\end{equation}
for $\epsilon>0$ sufficiently small.  Plugging \eqref{e:gepsilonsuper1}, \eqref{e:gepsilonsuper2}, \eqref{e:gepsilonsuper3} into \eqref{e:gepsilonsuper}, we get that 
\begin{equation}
\partial_t g_\epsilon (t,x) < \int\limits_{\R^{d-1}} \frac{\delta_h g_\epsilon(t,x) - \nabla g_\epsilon(t,x)\cdot h}{(\delta_h g_\epsilon(t,x)^2+|h|^2)^{d/2}} dh,
\end{equation}
for $(t,x)\in [\delta, T]\times \R^{d-1}$, completing the proof.  
\end{proof}

\begin{corollary}\label{c:growthbound}(Growth Bounds)

Let $f:[0,T]\times \R^{d-1}\to \R$ be a classical solution to the Muskat equation with $||\nabla_x f_0||_{L^\infty} < \displaystyle\frac{1}{\sqrt{2d-1}}$.  Then for any $h\in \R^{d-1}$, the function 
\begin{equation}\label{e:increments}
t\to\sup\limits_{x\in \R^{d-1}} f(t,x+h)-f(t,x),
\end{equation}
is nonincreasing.  In particular, 
\begin{equation}\label{e:growthbound}
\sup\limits_{x}\max\limits_{|h|\leq R} |f_0(x+h)-f_0(x)| \leq \Omega(R) \qquad \Rightarrow \qquad \sup\limits_{x}\max\limits_{|h|\leq R} |f(t,x+h)-f(t,x)| \leq \Omega(R).
\end{equation}
\end{corollary}

\begin{proof}
By Proposition \ref{p:maximumprinciple} $||\nabla_x f_0||_{L^\infty}< \displaystyle\frac{1}{\sqrt{2d-1}} \Rightarrow ||\nabla_x f||_{L^\infty}(t)\leq \displaystyle\frac{1}{\sqrt{2d-1}}$.  Thus Theorem \ref{t:comparison2} applies.  

Note that solutions to the Muskat equation are closed under translations in $\R^d$.  Hence, if $(t,x)\to f(t,x)$ is a solution, then so is $(t,x)\to f(t,x+h)+C$ for any fixed $h\in \R^{d-1}$ and $C\in \R$.  Applying Theorem \ref{t:comparison} to f and $g(t,x) = f(t,x+h)+C$ for appropriately chosen $C$ thus implies \eqref{e:increments}.  

Taking the supremum of \eqref{e:increments} over $|h|\leq R$ then gives \eqref{e:growthbound}.  
\end{proof}

\begin{corollary}(Uniqueness)\label{c:uniqueness}
Let $f_0,g_0\in \dot{W}^{1,\infty}$ with $||\nabla_x f_0||_{L^\infty},||\nabla_x g_0||_{L^\infty} < \displaystyle\frac{1}{\sqrt{2d-1}}$.  Then if $f,g:[0,T]\times \R^{d-1}\to \R$ are two classical solutions to the Muskat equation \eqref{e:fequation}, then 
\begin{equation}
\sup\limits_{x\in\R^{d-1}} f(t,x)-g(t,x)\leq \sup\limits_{x\in\R^{d-1}}f_0(x)-g_0(x).
\end{equation}
In particular, $f_0 \equiv g_0 \ \Rightarrow \ f\equiv g$.  
\end{corollary}

\begin{proof}
Proposition \ref{p:maximumprinciple} implies that if $f,g$ are two classical solutions of the Muskat equation with initial data $f_0.g_0$ satisfying 
\begin{equation}
||\nabla_x f_0||_{L^\infty},||\nabla_x g_0||_{L^\infty} <\frac{1}{\sqrt{2d-1}},
\end{equation}
then the same is true for later times $t>0$.  Taking advantage of the fact that solutions to the Muskat equation are closed under the addition of a constant $C\in \R$, we thus have by Theorem \ref{t:comparison} that 
\begin{equation}
f_0(x)\leq g_0(x) + \sup\limits_{y\in \R^{d-1}}f_0(y)-g_0(y) \qquad \Rightarrow \qquad f(t,x) \leq g(t,x) +  \sup\limits_{y\in \R^{d-1}}f_0(y)-g_0(y).
\end{equation}
\end{proof}

Under an assumption of uniformly bounded slope, it is possible to extend the proof of the comparison principle Theorem \ref{t:comparison2} to work for test functions with suitably small slope.  

\begin{proposition}\label{p:comparisongeneral}
Let $f:[0,T]\times \R^{d-1} \to \R$ be a classcial subsolution to the Muskat equation with $||\nabla_x f||_{L^\infty}\leq B$ for some fixed $B<\infty$.  Then there exists a constant $c(B)>0$ such that if $g:[0,T]\times \R^{d-1}\to \R$ is a classical supersolution to Muskat with $||\nabla_x g||_{L^\infty} < c(B)$ and $f(0,x)\leq g(0,x)$, we have
\begin{equation}
f(t,x)\leq g(t,x) \qquad \forall (t,x)\in [0,T]\times \R^{d-1}.
\end{equation}
Furthermore, for $B$ sufficiently large, we may take the constant $c(B) = \displaystyle\frac{B}{2(B^2+1)^{d/2}}$.  
\end{proposition}

Proposition \ref{p:comparisongeneral} allows for the possibility of employing barrier arguments to control the long range properties of solutions, so long as we assume an a priori bound on the slope.  Taking $g(t,x) \equiv \max f_0$ is one trivial example, implying the $L^\infty$ maximum principle for the interface $f$.  

\begin{proof}
For ease of proof, we assume $B>>1$.  In particular, $B>\displaystyle\frac{1}{\sqrt{2d-1}}$.  Following the argument in the proof of Theorem \ref{t:comparison2}, it suffices to show that if $(t,x)\in (0,T)\times \R^{d-1}$ is a first crossing point, then 
\begin{equation}\label{e:compintinequal}
\int\limits_{\R^{d-1}} \frac{\delta_hf(x) + \nabla_x f(x)\cdot h}{(\delta_hf(x)^2+|h|^2)^{d/2}}dh \leq\int\limits_{\R^{d-1}} \frac{\delta_hg(x) + \nabla_x g(x)\cdot h}{(\delta_hg(x)^2+|h|^2)^{d/2}}dh.
\end{equation}
In particular, it would suffice to show for every $h\not = 0$ that the we have 
\begin{equation}
 \frac{\delta_hf(x) + \nabla_x f(x)\cdot h}{(\delta_hf(x)^2+|h|^2)^{d/2}} \leq \frac{\delta_hg(x) + \nabla_x g(x)\cdot h}{(\delta_hg(x)^2+|h|^2)^{d/2}}.
\end{equation}
Note that $\nabla_x f(x) = \nabla_x g(x)$ at the crossing point.  Multiplying both sides by $|h|^{d-1}$, letting $t = \displaystyle\frac{\delta_hf(x)}{|h|}, s = \displaystyle\frac{\delta_hg(x)}{|h|}, a = \nabla_xf(x)\cdot \frac{h}{|h|} = \nabla_x g(x) \cdot \frac{h}{|h|}$, it suffices to prove that there exists a constant $c(B)$ such that 
\begin{equation}\label{e:compinequaltoprove}
\frac{t-a}{(t^2+1)^{d-2}} \leq \frac{s-a}{(s^2+1)^{d/2}}, \qquad \forall t\leq s, \quad t\in [-B,B], \quad s,a\in (-c(B),c(B)).
\end{equation}
We can of course assume that $c(B)\leq \displaystyle\frac{1}{\sqrt{2d-1}}$.  It then follows  that whenever $|t|<\displaystyle\frac{1}{\sqrt{2d-1}}$ by Lemma \ref{l:comparison} that 
\begin{equation}
\frac{t-a}{(t^2+1)^{d-2}} \leq \frac{s-a}{(s^2+1)^{d/2}}, \qquad \forall t\leq s, \quad t>-\displaystyle\frac{1}{\sqrt{2d-1}}, \quad s,a\in (-c(B),c(B)).
\end{equation}
We thus just need to prove the inequality in the case that $t\in [-B, \displaystyle\frac{1}{\sqrt{2d-1}}]$, $s,a\in (-c(B),c(B))$.  As $B>>1$ by assumption, it suffices to consider the extremal case $t=-B, s=-c(B), a=c(b)$.  Thus we just need to take $c(B)$ so that 
\begin{equation}
\frac{-B-c(B)}{(B^2+1)^{d/2}} \leq \frac{-2c(B)}{(c(B)^2+1)^{d/2}}, 
\end{equation}
As 
\begin{equation}
\frac{B+c(B)}{(B^2+1)^{d/2}}  \geq \frac{B}{(B^2+1)^{d/2}}, \qquad \frac{2c(B)}{(c(B)^2+1)^{d/2}}\leq 2c(B),
\end{equation}
taking $c(B) = \displaystyle\frac{B}{2(B^2+1)^{d/2}} $suffices.  Thus for this choice of constant $c(B)$, \eqref{e:compinequaltoprove} holds.  Hence, we have \eqref{e:compintinequal} for any crossing point $(t,x)\in (0,T)\times \R^{d-1}$.  Repeating the argument of Theorem \ref{t:comparison2}, we can guarantee the existence of a crossing point for an arbitrarily small perturbation of $g$.  Hence, Proposition \ref{p:comparisongeneral} follows.  
\end{proof}

%%%%%%%%%%%%%%%%%%%%%%%%%%%%%%%%%%%%%%%%%%%%%%%%%%%%%%%%%%%%%%%%%%%%%%%

\section{Regularity over time}\label{s:regularitytime}

With the construction of the modulus $\overline{\omega}$, we get universal Lipschitz bounds in space for $\nabla_x f(t,\cdot)$.  By the structure of \eqref{e:fequation}, we also get regularity in space for $\partial_t f$.
\begin{proposition}\label{p:regularityspace}
Let $f:[0,T]\times \R^{2} \to \R$ be a classical solution to \eqref{e:fequation} with $||\nabla_x f_0||_{L^\infty}<\displaystyle\frac{1}{\sqrt{5}}$, $||D^2_x f||_{L^\infty}(t)\lesssim 1/t.$, and growth bounds 
\begin{equation}
\sup\limits_{x} \max\limits_{|h|\leq R} |f_0(x+h)-f_0(x)|\leq \Omega(R), \qquad \int\limits_{1}^\infty  \frac{\Omega(R)}{R^2}dR < \infty.
\end{equation} 
Then $\partial_t f(t,\cdot)$ is Log-Lipschitz in space with
\begin{equation}
\begin{split}
||\partial_t f  ||_{L^\infty}(t) &\lesssim \inf\limits_{M>0} \frac{M}{t} + \int\limits_{M}^\infty \frac{\Omega(R)}{R^2}dR 
\\ |\partial_t f(t,x) -\partial_t f(t,y)| &\lesssim \frac{|x-y|}{t}\left(-\log\left(\frac{|x-y|}{t}\right) +\int\limits_{t}^\infty \frac{\Omega(R)}{R^2}dR\right), \qquad |x-y|\leq \frac{t}{2}.
\end{split}
\end{equation}
In particular, we have for any such solution $f$ and $0<t<1/2$ that 
\begin{equation}
\begin{split}
||\partial_t f ||_{L^\infty}(t)&\lesssim -\log(t) + \int\limits_{1}^\infty \frac{\Omega(R)}{R^2}dR,
\\ |\partial_t f(t,x)-\partial_t f(t,y)|&\lesssim \frac{|x-y|}{t}\left(-\log\left(\frac{|x-y|}{t}\right) -\log(t)+ \int\limits_1^\infty \frac{\Omega(R)}{R^2}dR \right), \qquad |x-y|\leq \frac{t}{2}.
\end{split}
\end{equation}
In the case that we can take $\Omega$ of the form $\Omega(R)=\Omega_0 R^{\alpha}$ for some $0\leq \alpha <1$, then we also get the large time decay 
\begin{equation}
||\partial_t f||_{L^\infty}(t) \lesssim \frac{\Omega_0^{\frac{1}{2-\alpha}}t^{-\frac{1-\alpha}{2-\alpha}}}{1-\alpha},
\end{equation}
\end{proposition}

\begin{proof}
Note that by Corollary \ref{c:growthbound}, we have for all $t>0$ that 
\begin{equation}
\sup\limits_{x\in \R^{2}}\max\limits_{|h|\leq R} |f(t,x+h)-f(t,x)|\leq \Omega(R).
\end{equation}
We have that
\begin{equation}\label{e:timederivativebound}
\begin{split}
|\partial_t f(t,x)| &= \bigg|\int\limits_{\R^{2}} \frac{\delta_hf(t,x) - \nabla_x f(t,x)\cdot h}{(\delta_hf(t,x)^2+|h|^2)^{3/2}} dh\bigg|  \leq \frac{1}{2}\bigg|\int\limits_{\R^{2}} \frac{\delta_hf(t,x) +\delta_{-h}f(t,x)}{(\delta_{-h}f(t,x)^2 +|h|^2)^{3/2}}dh\bigg|
\\&\qquad+ \frac{1}{2}\bigg|\int\limits_{\R^{2}} \frac{(\delta_hf(t,x)-\nabla_x f(t,x)\cdot h)\left[(\delta_hf(t,x)^2+|h|^2)^{3/2}-(\delta_{-h}f(t,x)^2+|h|^2)^{3/2}\right]}{(\delta_hf(t,x)^2+|h|^2)^{3/2}(\delta_{-h}f(t,x)^2 +|h|^2)^{3/2}} dh\bigg|
\end{split}
\end{equation}
We bound each of the above integrals in two ways.  For small $h$, we use our second derivative bounds to get 
\begin{equation}\label{e:timesmallhbound}
\begin{split}
 \bigg|\frac{\delta_hf(t,x) +\delta_{-h}f(t,x)}{(\delta_{-h}f(t,x)^2 +|h|^2)^{3/2}}\bigg| \lesssim \frac{|h|^2}{t}\frac{1}{|h|^{3}}= \frac{1}{t}\frac{1}{|h|},
\\ \bigg| \frac{(\delta_hf(t,x)-\nabla_x f(t,x)\cdot h)\left[(\delta_hf(t,x)^2+|h|^2)^{3/2}-(\delta_{-h}f(t,x)^2+|h|^2)^{3/2}\right]}{(\delta_hf(t,x)^2+|h|^2)^{3/2}(\delta_{-h}f(t,x)^2 +|h|^2)^{3/2}}\bigg| \lesssim \frac{|h|^2}{t} \frac{|h|^{3}}{|h|^{6}} = \frac{1}{t}\frac{1}{|h|}.
\end{split}
\end{equation}
Similarly for large $h$, we can use our growth bounds to get 
\begin{equation}\label{e:timelargehbound}
\begin{split}
 \bigg|\frac{\delta_hf(x) +\delta_{-h}f(x)}{(\delta_{-h}f(x)^2 +|h|^2)^{3/2}}\bigg| &\lesssim \frac{\Omega(|h|)}{|h|^{3}},
\\ \bigg| \frac{(\delta_hf(x)-\nabla_x f(x)\cdot h)\left[(\delta_hf(x)^2+|h|^2)^{3/2}-(\delta_{-h}f(x)^2+|h|^2)^{3/2}\right]}{(\delta_hf(x)^2+|h|^2)^{3/2}(\delta_{-h}f(x)^2 +|h|^2)^{3/2}}\bigg| &\lesssim \frac{|h|}{|h|^{6}} \bigg|\delta_h f(x)-\delta_{-h}f(x)\bigg| \cdot |h|^{2} 
\\&\lesssim \frac{\Omega(|h|)}{|h|^3}
\end{split}
\end{equation}
Hence for any $M>0$, plugging in \eqref{e:timesmallhbound} for $|h|\leq M$ and \eqref{e:timelargehbound} for $|h|\geq M$ into \eqref{e:timederivativebound} gives us 
\begin{equation}
|\partial_t f(t,x)| \lesssim \frac{M}{t} + \int\limits_{M}^\infty  \frac{\Omega(R)}{R^2}dR.  
\end{equation}
%\begin{equation}
%\begin{split}
%|\partial_t f(t,x)|&
% \lesssim \int\limits_{|h|\leq t} \frac{1}{t|h|^{d-2}}dh  + \int\limits_{t\leq |h|\leq 1}\frac{1}{|h|^{d-1}}dh   +  \int\limits_{|h|>1}\frac{\Omega(|h|)}{|h|^d}dh \\&\lesssim -\log(t) + 1.
%\end{split}
%\end{equation}
In particular for $t<\displaystyle\frac{1}{2}$, we may take $M = t$ and $\Omega(R) = ||\nabla_x f||_{L^\infty}R $ for $R<1$ to get 
\begin{equation}
|\partial_t f(t,x)| \lesssim 1 + \int\limits_{t}^1 \frac{1}{R}dR + \int\limits_{1}^\infty  \frac{\Omega(R)}{R^2}dR \lesssim -\log(t) + \int\limits_{1}^\infty \frac{\Omega(R)}{R^2}dR, \qquad 0<t<\frac{1}{2}.
\end{equation}
If $\Omega(R)$ is of the form $\Omega(R) = \Omega_0 R^{\alpha}$ for some $0\leq \alpha<1$, then taking $M = (\Omega_0 t)^{1/(2-\alpha)}$ gives
\begin{equation}
|\partial_t f(t,x)| \lesssim \frac{\Omega_0^{\frac{1}{2-\alpha}}t^{-\frac{1-\alpha}{2-\alpha}}}{1-\alpha},
\end{equation}
which is useful for large times $t$.

For regularity in space, fix some time $t>0$ and $x,y\in\R^{2}$ with $|x-y|\leq\displaystyle\frac{t}{2}$.  Then 
\begin{equation}
\begin{split}
|\partial_t f(t,x)  - \partial_t f(t,xy)| &= \bigg|\int\limits_{\R^{2}} \frac{\delta_h f(x) - \nabla_x f(x)\cdot h}{(\delta_h f(x)^2 + |h|^2)^{3/2}} -\frac{\delta_h f(y) - \nabla_x f(y)\cdot h}{(\delta_h f(y)^2 + |h|^2)^{3/2}} dh\bigg| 
\\&\leq \bigg|\int\limits_{|h|\leq |x-y|}\bigg| + \bigg| \int\limits_{|x-y|\leq |h|\leq t} \bigg| + \bigg| \int\limits_{|h|\geq t} \bigg| := I_1+I_2+I_3.
\end{split}
\end{equation}
For small $|h|\leq |x-y|$, it is best to bound each term separately with our second derivative bounds.  Thus similarly to above,
\begin{equation}
\begin{split}\label{e:timespaceregI1}
 \bigg|\int\limits_{|h|\leq |x-y|} \frac{\delta_h f(x) - \nabla_x f(x)\cdot h}{(\delta_h f(x)^2 + |h|^2)^{3/2}} dh\bigg|, \ \bigg|\int\limits_{|h| \leq |x-y|}\frac{\delta_h f(y) - \nabla_x f(y)\cdot h}{(\delta_h f(y)^2 + |h|^2)^{3/2}} dh\bigg| \lesssim \frac{|x-y|}{t},
 \\ \Rightarrow I_1 \lesssim \frac{|x-y|}{t}.
 \end{split}
\end{equation}

For midsize $|x-y|\leq |h|\leq t$, we split the integral into two pieces: 
\begin{equation}\label{e:midhfirstbound}
\begin{split}
I_2 & \leq \bigg| \int\limits_{ |x-y|\leq |h|\leq t} \frac{\delta_h f(t,x) -\nabla_x f(t,x)\cdot h - (\delta_hf(t,y)-\nabla_x f(t,y)\cdot h)}{(\delta_h f(t,y)^2+|h|^2)^{3/2}}dh\bigg|
\\& \ +\bigg|\int\limits_{|x-y|\leq |h|\leq t} \frac{(\delta_hf(t,x) -  \nabla_x f(t,x)\cdot h)\left[(\delta_hf(t,x)^2+|h|^2)^{3/2}-(\delta_hf(t,y)^2+|h|^2)^{3/2}\right])}{( \delta_hf(t,x)^2 + |h|^2)^{3/2}(\delta_hf(t,y)^2 + |h|^2)^{3/2}} dh\bigg|.
\end{split}
\end{equation}
Using our Lipschitz bounds on $\nabla_x f$, we get the bounds
\begin{equation}\begin{split}
\bigg|\delta_h f(t,x) - \nabla_x f(t,x)\cdot h - (\delta_hf(t,y)-\nabla_x f(t,y)\cdot h)\bigg| = \bigg|\left(\int\limits_0^1 \delta_{sh}\nabla_x f(t,x)- \delta_{sh} \nabla_x f(t,y)ds\right)\cdot h \bigg|\lesssim \frac{|x-y| \ |h|}{t},
\\ \bigg| \delta_hf(t,x) - \delta_hf(t,y)\bigg| = \bigg| \left(\int\limits_0^1 \nabla_x f(t,x+sh) - \nabla_x f(t,y+sh)ds\right)\cdot h\bigg|\lesssim \frac{|x-y|\ |h|}{t}.
\end{split}
\end{equation} 
Plugging these into \eqref{e:midhfirstbound} gives 
\begin{equation}
\begin{split}\label{e:timespaceregI2}
I_2 &\lesssim   \int\limits_{ |x-y|\leq |h|\leq t} \frac{\bigg|\delta_h f(x) - \nabla_x f(x)\cdot h - (\delta_hf(y)-\nabla_x f(y)\cdot h)\bigg| }{|h|^3}  +\frac{|h|}{|h|^{6}} \left[|\delta_h f(x)-\delta_hf(y)||h|^{4}\right] dh 
\\&\lesssim \frac{|x-y|}{t}\int\limits_{|x-y|\leq |h|\leq t}\frac{1}{|h|^{2}}dh = -\frac{|x-y|}{t}\log\left(\frac{|x-y|}{t}\right).
\end{split}
\end{equation}
Finally, to bound large $|h|\geq t$, we split our integral as 
\begin{equation}
\begin{split}\label{e:largehfirstbound}
I_3 &\leq \bigg|\int\limits_{|h|>t}  \frac{\delta_h f(x) - \delta_hf(y)}{(\delta_h f(y)^2+|h|^2)^{3/2}} + \frac{(\delta_hf(x) - \nabla_x f(x)\cdot h)\left[(\delta_hf(x)^2+|h|^2)^{3/2}-(\delta_hf(y)^2+|h|^2)^{3/2}\right]}{( \delta_hf(x)^2 + |h|^2)^{3/2}(\delta_hf(y)^2 + |h|^2)^{3/2}} dh\bigg|
\\&\qquad +|\nabla_x f(x)-\nabla_x f(y)| \ \bigg|\int\limits_{|h|>t} \frac{-h}{(\delta_hf(y)^2 +|h|^2)^{3/2}} dh \bigg|
\end{split}
\end{equation}
Similar to the mid $h$ bound, we can use our Lipschitz bounds on $f$ to get that 
\begin{equation}\label{e:largehfirstint}
\begin{split}
 \bigg|\int\limits_{|h|>t}  \frac{\delta_h f(x) - \delta_hf(y)}{(\delta_h f(y)^2+|h|^2)^{3/2}} &+ \frac{(\delta_hf(x) - \nabla_x f(x)\cdot h)\left[(\delta_hf(x)^2+|h|^2)^{3/2}-(\delta_hf(y)^2+|h|^2)^{3/2}\right]}{( \delta_hf(x)^2 + |h|^2)^{3/2}(\delta_hf(y)^2 + |h|^2)^{3/2}} dh\bigg|
 \\& \lesssim \int\limits_{|h|>t} \frac{|\delta_h f(x)-\delta_h f(y)|}{|h|^3} dh \lesssim \frac{|x-y|}{t}.
\end{split}
\end{equation}
For the second integral in \eqref{e:largehfirstbound}, we use our growth bounds on $f$ to get that
\begin{equation}
\begin{split}\label{e:largehsecondint}
\bigg|\int\limits_{|h|>t} & \frac{-h}{(\delta_hf(y)^2 +|h|^2)^{3/2}} dh \bigg| = \frac{1}{2}\bigg| \int\limits_{|h|>t}  \frac{-h}{(\delta_hf(y)^2 +|h|^2)^{3/2}}  +  \frac{h}{(\delta_{-h}f(y)^2 +|h|^2)^{3/2}} dh \bigg|
\\&\lesssim  \bigg| \int\limits_{|h|>t} \frac{|\delta_h f(y)- \delta_{-h}f(y)|}{|h|^3}dh \bigg| \lesssim \int\limits_{t}^\infty \frac{\Omega(R)}{R^2}dR.
\end{split}
\end{equation}
Plugging \eqref{e:largehfirstint} and \eqref{e:largehsecondint} into \eqref{e:largehfirstbound}, and using that $|\nabla_x f(t,x)-\nabla_xf(t,y)|\lesssim \displaystyle\frac{|x-y|}{t}$ gives 
\begin{equation}\label{e:timespaceregI3}
I_3 \lesssim \frac{|x-y|}{t}\left(1+\int\limits_t^\infty \frac{\Omega(R)}{R^2}dR.\right).
\end{equation}

Combining the various bounds \eqref{e:timespaceregI1}, \eqref{e:timespaceregI2}, \eqref{e:timespaceregI3} and using that $1\lesssim -\log\left(\displaystyle\frac{|x-y|}{t}\right)$ thus gives us 
\begin{equation}
|\partial_t f(t,x) -\partial_t f(t,y)|\lesssim \frac{|x-y|}{t}\left(-\log\left(\frac{|x-y|}{t}\right) + \int\limits_t^\infty \frac{\Omega(R)}{R^2}dR\right), \qquad |x-y|\leq \frac{t}{2}.
\end{equation}

For small times $0 < t\leq 1$, by taking $\Omega(R) = ||\nabla_x f||_{L^\infty}R$ for $R\leq 1$ we get that 
\begin{equation}
\begin{split}
|\partial_t f(t,x) -\partial_t f(t,y)|&\lesssim \frac{|x-y|}{t}\left(-\log\left(\frac{|x-y|}{t}\right) + \int\limits_t^\infty \frac{\Omega(R)}{R^2}dR\right)
\\&\lesssim  \frac{|x-y|}{t}\left(-\log\left(\frac{|x-y|}{t}\right) -\log(t)+ \int\limits_1^\infty \frac{\Omega(R)}{R^2}dR\right), \qquad |x-y|\leq \frac{t}{2}.
\end{split}
\end{equation}

\end{proof}

\begin{proposition}
Let $f:[0,T]\times \R^{2} \to \R$ be a classical solution to \eqref{e:fequation} with $||\nabla_x f_0||_{L^\infty}<\displaystyle\frac{1}{\sqrt{5}}$, and growth bounds 
\begin{equation}
\sup\limits_{x} \max\limits_{|h|\leq R} |f_0(x+h)-f_0(x)|\leq \Omega(R), \qquad \int\limits_{1}^\infty  \frac{\Omega(R)}{R^2}dR < \infty.
\end{equation} 
Without loss of generality, assume that $\Omega(R)$ is concave.  Then letting , $\overline{\Omega}(R) = \displaystyle\frac{\Omega(R)}{R}$, which is nonincreasing, we have that 
\begin{equation}
||\nabla_x f||_{L^\infty}(t)\leq 2\overline{\Omega}\left(
\frac{2||\nabla_x f||_{L^\infty}(t)}{||D^2_x f||_{L^\infty}(t)}\right).
\end{equation}
In particular, $\displaystyle\lim\limits_{t\to \infty} ||\nabla_x f||_{L^\infty}(t) = 0$ uniformly depending only on $\Omega, ||\nabla_x f_0||_{L^\infty}$, and dimension $d$.  In the case that $\Omega(R) = \Omega_0 R^{\alpha}$ for some $0\leq \alpha <1$, then 
\begin{equation}
||\nabla_x f||_{L^\infty}(t) \lesssim \frac{\Omega_0^{\frac{1}{2-\alpha}}}{t^{\frac{1-\alpha}{2-\alpha}}}.
\end{equation}
\end{proposition}

\begin{proof}
Again by Corollary \ref{c:growthbound} we have for all $t>0$ that 
\begin{equation}
\sup\limits_{x\in \R^{2}}\max\limits_{|h|\leq R} |f(t,x+h)-f(t,x)|\leq \Omega(R).
\end{equation}
Fix some time $t>0$ and $x\in \R^{2}$.  Without loss of generality, assume that $\nabla_x f(t,x)\not = 0$.  Taking
$R = R(t,x) = \displaystyle\frac{|\nabla_x f(t,x)|}{||D_x^2 f||_{L^\infty}(t)}$, we then have that 
\begin{equation}
\frac{|\nabla_x f(t,x)|^2}{||D^2_x f||_{L^\infty}(t)}  = |\nabla_x f(t,x)| R  = 2 |\nabla_x f(t,x)| R - ||D_x^2 f||_{L^\infty}(t)R^2 \leq \osc_{B_R(x)} f(t,\cdot) \leq \Omega(2R).
\end{equation}
Rearranging, we get 
\begin{equation}
|\nabla_x f(t,x)| \leq 2\frac{\Omega(2R)}{2R} = 2\overline{\Omega}\left(
\frac{2|\nabla_x f(t,x)|}{||D^2_x f||_{L^\infty}(t)}\right).
\end{equation}
Taking the supremum in $x\in \R^{2}$, we thus have that 
\begin{equation}\label{e:slopeupperbound}
||\nabla_x f||_{L^\infty}(t)\leq 2\overline{\Omega}\left(
\frac{2||\nabla_x f||_{L^\infty}(t)}{||D^2_x f||_{L^\infty}(t)}\right).
\end{equation}
Note that since $\Omega:[0,\infty)\to [0,\infty)$ is concave with $\displaystyle\int\limits_1^\infty \frac{\Omega(R)}{R^2}dR <\infty$, it follows that $R \to \overline{\Omega}(R)$ is nonincreasing with $\displaystyle\lim\limits_{R\to \infty} \overline{\Omega}(R) = 0$.  Thus \eqref{e:slopeupperbound} implicitly gives an upperbound on the slope $||\nabla_x f||_{L^\infty}(t)$ in terms of $\Omega$ and $||D_x^2 f||_{L^\infty}(t)$.  As $\displaystyle\lim\limits_{t\to \infty} ||D^2_x f||_{L^\infty}(t) = 0$ with a rate only depending on $||\nabla_x f_0||_{L^\infty}$ and dimension, and $\overline{\Omega}(R) \to 0$ as $R\to \infty$ with a rate depending only on $\Omega$, we thus have that 
\begin{equation}
\lim\limits_{t\to \infty} ||\nabla_x f||_{L^\infty} (t) = 0,
\end{equation}
with a rate depending only on those quantities as well.  

Finally, in the case that $\Omega(R) = \Omega(R) = \Omega_0 R^\alpha$ for some $0\leq \alpha<1$, then using that $||D^2_x f||_{L^\infty}(t) \lesssim \displaystyle\frac{1}{t}$ and rearranging \eqref{e:slopeupperbound} gives 
\begin{equation}
||\nabla_x f||_{L^\infty}(t) \lesssim \frac{\Omega_0^{\frac{1}{2-\alpha}}}{t^{\frac{1-\alpha}{2-\alpha}}}.
\end{equation}
\end{proof}

\begin{proposition}\label{p:C1alphabound}
Let $f:[0,T]\times \R^{2} \to \R$ be a classical solution to \eqref{e:fequation} with $||\nabla_x f_0||_{L^\infty}<\displaystyle\frac{1}{\sqrt{5}}$, and growth bounds 
\begin{equation}
\sup\limits_{x} \max\limits_{|h|\leq R} |f_0(x+h)-f_0(x)|\leq \Omega(R), \qquad \int\limits_{1}^\infty  \frac{\Omega(R)}{R^2}dR < \infty.
\end{equation} 
Then $\nabla_x f \in C^{\alpha}_{loc}((0,T]\times \R^{2}; \R^{2})$ and $\partial_t f\in C^{\alpha}_{loc}((0,T]\times \R^{2})$ with
\begin{equation}
||\nabla_x f||_{C^{\alpha}(Q_{t/4}(t,x))}, ||\partial_t f||_{C^{\alpha}(Q_{t/4}(t,x))}\leq C(||\nabla_x f_0||_{L^\infty}, \Omega, ||D^2_x f ||_{L^\infty((t/2,3t/2)\times \R^{2})}) \max\{t^{-\alpha},1\},
\end{equation}
where $Q_r(s,y) = (s-r,s]\times B_r(y)$, and $\alpha>0$ depends only on $||\nabla_x f_0||_{L^\infty}$ and dimension.
\end{proposition}

\begin{proof}

Let $e\in S^{1}$ be arbitrary.  Then we have that $f_e$ satisfies the equation 
\begin{equation}
\partial_t f_e (t,x) =  \nabla_x f_{e}(t,x)\cdot  \int\limits_{\R^{2}}\frac{-h}{(\delta_hf(t,x)^2+|h|^2)^{3/2}} dh + \int\limits_{\R^{2}} \delta_hf_e(t,x) K(t,x,h) dh,
\end{equation}
where $K$ as defined in \eqref{e:Kdefn}.  Since $||\nabla_x f_0||_{L^\infty}<\displaystyle\frac{1}{\sqrt{5}}$, it follows by proposition \ref{p:maximumprinciple} that $K$ is uniformly elliptic with ellipticity constants depending only on $||\nabla_x f_0||_{L^\infty}$.  Rewriting this equation slightly, we have that 
\begin{equation}\label{e:gequationsymmetric}
\begin{split}
\partial_t f_e - \int\limits_{\R^{2}} \delta_h f_e(t,x) \left(\frac{K(t,x,h)+K(t,x,-h)}{2}\right) dh &=  \nabla_x f_{e}(t,x) \cdot \int\limits_{\R^{2}}\frac{-h}{(\delta_hf(t,x)^2+|h|^2)^{3/2}} dh
\\&\qquad+ \int\limits_{\R^{2}} \delta_h f_e(t,x) \left(\frac{K(t,x,h)-K(t,x,-h)}{2}\right) dh.
\end{split}
\end{equation}

Let $F(t,x)$ denote the righthand side of \eqref{e:gequationsymmetric}.  Then we claim that $|F(t,\cdot)|$ is bounded uniformly in terms of $||D_x^2 f||_{L^\infty}(t)$ and the growth rate $\Omega$.  To see this, note that we can bound the drift term 
\begin{equation}\label{e:driftLinftybound}
\begin{split}
\bigg|\int\limits_{\R^{2}}\frac{-h}{(\delta_hf(t,x)^2+|h|^2)^{3/2}} dh \bigg|&= \bigg|\frac{1}{2}\int\limits_{\R^{2}}  h \frac{(\delta_hf(t,x)^2+|h|^2)^{3/2} - (\delta_{-h}f(t,x)^2+|h|^2)^{3/2}}{(\delta_hf(t,x)^2 +|h|^2)^{3/2}(\delta_{-h}f(t,x)^2 +|h|^2)^{3/2}} dh\bigg|
\\&\lesssim \int\limits_{\R^{d-1}} \frac{|\delta_h f(t,x)+\delta_{-h}f(t,x)|}{|h|^3}dh 
\\&\lesssim \int\limits_{|h|\leq 1} \frac{||D^2_x f||_{L^\infty}(t)}{|h|}dh + \int\limits_{|h|\geq 1} \frac{\Omega(|h|)}{|h|^3}dh 
\\&\lesssim ||D^2_x f||_{L^\infty}(t) + \int\limits_1^\infty \frac{\Omega(R)}{R^2}dR.
\end{split}
\end{equation}
Similarly, the asymmetry bounds on $K$ proven in Lemma \ref{l:asymmetry} combined with the ellipticity bounds give us 
\begin{equation}\label{e:asymetricLinftybound}
\begin{split}
\bigg|\int\limits_{\R^{2}} \delta_h f_e(t,x) \left(\frac{K(t,x,h)-K(t,x,-h)}{2}\right) dh \bigg| &\leq  \bigg|\int\limits_{|h|\leq 1} \delta_h f_e(t,x) \left(\frac{K(t,x,h)-K(t,x,-h)}{2}\right) dh\bigg| 
\\& \ + \bigg|\int\limits_{|h|\geq 1} \delta_h f_e(t,x) \left(\frac{K(t,x,h)-K(t,x,-h)}{2}\right) dh\bigg|
\\&\lesssim \int\limits_{|h|\leq 1} \frac{||D_x^2 f||_{L^\infty}(t)^2 }{|h|} dh  + \int\limits_{|h|\geq 1} \frac{1}{|h|^3} dh 
\\& \lesssim ||D^2_x f||_{L^\infty}(t)^2 + 1.
\end{split}
\end{equation}
Combining \eqref{e:driftLinftybound} and \eqref{e:asymetricLinftybound}, we thus have that our directional derivative $f_e$ solves the equation 
\begin{equation}
\begin{split}
\partial_t f_e(t,x) - \int\limits_{\R^{2}} \delta_h f_e(t,x) \left(\frac{K(t,x,h)+K(t,x,-h)}{2}\right) dh = F(t,x), \qquad (t,x)\in (0,T]\times \R^{2}, 
\\ \frac{\lambda}{|h|^d} \leq \frac{K(t,x,h)-K(t,x,-h)}{2}\leq \frac{\Lambda}{|h|^3}, \qquad |F(t,x)|\lesssim 1+ ||D^2_x f||_{L^\infty}^2(t) + \int\limits_1^\infty \frac{\Omega(R)}{R^2}dR.
\end{split}
\end{equation}

Thus our directional derivative $f_e$ solves a parabolic equation with uniformly elliptic, symmetric kernel $\displaystyle\frac{K(t,x,h)+K(t,x,-h)}{2}$ and right hand side $F\in L^\infty([\epsilon, T]\times \R^{2})$ for positive $\epsilon>0$.  Rescaling and applying the results of \cite{SilvestreHolder}, we thus have that there is an $\alpha>0$ depending only on the ellipticity constants and dimension such that for any $(t,x)\in (0,T]\times \R^{2}$, 
\begin{equation}
 || f_e ||_{C^{\alpha}(Q_{t/4}(t,x))}\leq C \max\{t^{-\alpha},1\}(||f_e||_{L^{\infty}(Q_{t/2}(t,x))} + ||F||_{L^\infty(Q_{t/2}(t,x))}),
\end{equation}
where $Q_r(s,y) = (s-r,s]\times B_r(y)$.  As $e\in S^{1}$ was arbitrary, we've thus proven the $C^{\alpha}$ bound for $\nabla_x f$.  

To get the $C^\alpha$ estimate for $\partial_t f$, we simply rely on the equation our original equation
\begin{equation}
\partial_t f(t,x)  = \int\limits_{\R^{2}} \frac{\delta_h f(t,x) - \nabla_x f(t,x)\cdot h}{(\delta_h f(t,x)^2+|h^2)^{3/2}}dh.
\end{equation}
Doing similar bounds as in Proposition \ref{p:regularityspace}, control over $||D^2_x f||_{L^\infty}$, $\Omega$, and a $C^\alpha$ estimate for $\nabla_x f$ that is uniform in space similarly gives a $C^\alpha$ estimate for $\partial_t f$.

\end{proof}

%%%%%%%%%%%%%%%%%%%%%%%%%%%%%%%%%%%%%%%%%%%%%%%%%%%%%%%%%%%%%%%%%%%%%%

\section{Existence of solutions for all time}\label{s:globalexistence}

So far, our results and a priori estimates have been for sufficiently regular, classical solutions $f:[0,T]\times \R^{2}\to \R$ to the Muskat equation \eqref{e:fequation}.  While we have formally only derived these for smooth solutions, the estimates only depend quantitatively on $||\nabla_x f_0||_{L^\infty}<\displaystyle\frac{1}{\sqrt{5}}$ and integrably sublinear growth bounds on $f_0$, i.e. 
\begin{equation}
\sup\limits_{x\in \R^{2}}\max\limits_{|h|\leq R} |f_0(x+h)-f_0(x)|\leq \Omega(R), \qquad \int\limits_1^\infty \frac{\Omega(R)}{R^2}dR < \infty.
\end{equation} 

Notably, none of these estimates depending on the time of existence $T$.  We shall now show that under these assumptions on the initial data $f_0$, there exists a unique classical solution $f:[0,\infty)\times \R^{2}\to \R$ to \eqref{e:fequation} with initial data $f_0$.  

To prove this, we first show that the same is true for solutions to an $\epsilon$-viscious Muskat equation, 
\begin{equation}\label{e:epsilonsystem}
\partial_t f(t,x) - \epsilon \Delta_x f(t,x) = \int\limits_{\R^{2}}\frac{\delta_h f(t,x) - \nabla_x f(t,x)\cdot h}{(\delta_h f(t,x)^2+|h|^2)^{3/2}}dh. 
\end{equation}

\begin{lemma}\label{l:epsilonaprioriestimates}
Let $f: [0,T]\times \R^{2}\to \R$ be a sufficiently regular solution of the $\epsilon$-viscous Muskat equation \eqref{e:epsilonsystem} with initial data $f_0$ satisfying $||\nabla_x f_0||_{L^\infty}<\displaystyle\frac{1}{\sqrt{5}}$ and growth bounds 
\begin{equation}
\sup\limits_{x\in \R^{2}}\max\limits_{|h|\leq R} |f_0(x+h)-f_0(x)|\leq \Omega(R), \qquad \int\limits_1^\infty \frac{\Omega(R)}{R^2}dR < \infty.
\end{equation} 
Then solutions to the $\epsilon$-system satisfy the comparison principle and hence the growth bounds of Corollary \ref{c:growthbound},
\begin{equation}
\sup\limits_{x\in \R^2} f(t,x+h)-f(t,x) \leq \sup\limits_{x\in \R^2} f_0(x+h)-f_0(x), \qquad \forall h\in \R^2,
\end{equation}
as well as the a priori estimates of Corollary \ref{c:C11bound}, Proposition \ref{p:regularityspace}, and Proposition \ref{p:C1alphabound}
\begin{equation}
\begin{split}
||D_x^2 f||_{L^\infty}(t)\lesssim \min\left\{||D_x^2 f_0||_{L^\infty}, \frac{1}{t}\right\},
\\
||\partial_t f - \epsilon \Delta_x f||_{L^\infty}(t)\lesssim -\log(t), \qquad t<\frac{1}{2},
\\
||\nabla_x f||_{C^{\alpha}(Q_{t/4}(t,x))}, ||\partial_t f||_{C^{\alpha}(Q_{t/4}(t,x))}\leq C(||\nabla_x f_0||_{L^\infty}, \Omega,t).  
\end{split}
\end{equation}
with constants independent of $\epsilon>0$.  
\end{lemma}

\begin{lemma}\label{l:epsilonglobalexistence}
Let $f_0\in C^{\infty}_c(\R^{2})$ be such that $||\nabla_x f_0||_{L^\infty}<\displaystyle\frac{1}{\sqrt{5}}$ and  
\begin{equation}
\sup\limits_{x\in \R^{2}}\max\limits_{|h|\leq R} |f_0(x+h)-f_0(x)|\leq \Omega(R), \qquad \int\limits_1^\infty \frac{\Omega(R)}{R^2}dR < \infty.
\end{equation} 
Then for any $\epsilon>0$, there exists a unique solution $f:[0,\infty)\times \R^{2}\to \R$ to the $\epsilon$-viscous Muskat equation \eqref{e:epsilonsystem}.  
\end{lemma}
\begin{proof}
From \cite{Localwellpose}, we have that the Muskat equation \eqref{e:fequation} is locally wellposed in $H^k(\R^{2})$ for $k\geq 4$.  Furthermore, they establish the continuation criteria that if $T = T(f_0,k)$ is the maximal time of existence for a solution $f(t,\cdot)\in H^k(\R^{2})$, then
\begin{equation}
T(f_0,k)<\infty \qquad \Rightarrow \qquad \lim\limits_{t\to T-} ||f||_{C^{2,\beta}(\R^{2})}(t) = \infty,
\end{equation}
for any $\beta>0$.  Their proof is all based on $L^2$-energy estimates for the Muskat equation.  In particular, all of these results still hold for the $\epsilon$-viscous Muskat equation.  

Thus for given $f_0\in C^\infty_c(\R^{2})$, we have that there is some maximal time $T = T(f_0, k)>0$ such that there exists a smooth solution $f: [0,T)\times \R^{2}\to \R$ to the $\epsilon$-viscous Muskat equation.  The goal of the lemma then is to show in fact $T = \infty$.  

By Lemma \ref{l:epsilonaprioriestimates}, we have that for any time $t>0$ and $x\in \R^{2}$ that 
\begin{equation}
||f||_{C^{1,\alpha}(Q_{t/4}(t,x))} \leq C(t, f_0)<\infty,
\end{equation}
for some $\alpha>0$ depending only on $||\nabla_x f_0||_{L^\infty}$.  The bound $C(t,f_0)$ can be taken to be nonincreasing in time $t>0$.  

Now consider the $\epsilon$-Muskat equation.  We have for any $0<\delta < T(f_0)$ that 
\begin{equation}
\partial_t f(t,x) - \epsilon \Delta_x f(t,x) = \left(\int\limits_{\R^{2}} \frac{\delta_h f(t,x) - \nabla_x f(t,x)\cdot h}{(\delta_h f(t,x)^2+|h|^2)^{3/2}} dh \right)\in C^{\alpha}([\delta, T)\times \R^{2}), 
\end{equation}
with $C^\alpha$ norm depending only on $\delta$ and $f_0$.  Thus $f$ solves the heat equation with a uniformly bounded $C^\alpha$ source term, so we get that 
\begin{equation}
||f||_{C^{2, \alpha}} \leq C(f_0, \delta, \alpha, \epsilon)<\infty.
\end{equation}
Though this $C^{2,\alpha}$ bound does depend on $\epsilon>0$, what is important is that the $C^{2,\alpha}$ norm of $f$ cannot blow up.  Thus by the continuation criteria, we have that in fact $T = T(f_0) = \infty$ proving the lemma.  
\end{proof}

\begin{lemma}\label{l:smoothglobalexistence}
Let $f_0\in C^{\infty}_c(\R^{2})$ satisfy the assumptions of Lemma \ref{l:epsilonglobalexistence}. Then there exists a global classical solution $f\in C([0,\infty)\times \R^2)\cap C^{1,\alpha}_{loc}((0,\infty)\times \R^2)\cap L^\infty ([0,\infty); C^{1,1}(\R^2))$ to the Muskat equation \eqref{e:fequation}.  
\end{lemma}
\begin{proof}
For every $\epsilon>0$, we have that there exists a global smooth solution $f^{(\epsilon)}$ to the $\epsilon$-regularized Muskat equation \eqref{e:epsilonsystem}.  By the a priori estimates given in Lemma \ref{l:epsilonaprioriestimates}, we have that the sequence $(f^{(\epsilon)})_{\epsilon>0}$ is precompact in $C^1_{loc}([0,\infty)\times \R^2)$.  Hence there is a sequence $\epsilon_k\to 0$ and $f\in C([0,\infty)\times \R^2)\cap C^{1,\alpha}_{loc}((0,\infty)\times \R^2)\cap L^\infty ([0,\infty); C^{1,1}(\R^2))$ such that $f^{(\epsilon_k)} \to f$ in $C^1_{loc}$.  As $f^{(\epsilon)}$ solve the $\epsilon$-system \eqref{e:epsilonsystem} with initial value $f^{(\epsilon)}(0,x) = f_0(x)$ for all $\epsilon>0$, we thus have that the limit $f$ solves the Muskat equation \eqref{e:fequation} with $f(0,x) = f_0(x)$ as well.  
\end{proof}

\begin{theorem}
Let $f_0\in \dot{W}^{1,\infty}(\R^2)$ satisfy $||\nabla_x f_0||_{L^\infty} < \displaystyle\frac{1}{\sqrt{5}}$ and the growth bounds 
\begin{equation}
\sup\limits_{x\in \R^{2}}\max\limits_{|h|\leq R} |f_0(x+h)-f_0(x)|\leq \Omega(R), \qquad \int\limits_1^\infty \frac{\Omega(R)}{R^2}dR < \infty.
\end{equation} 
Then there exists a unique classical solution $f:[0,\infty)\times \R^{2}\to \R$ to the Muskat equation \eqref{e:fequation} satisfying the previous estimates of sections \ref{s:modulusrho} and \ref{s:regularitytime}.
\end{theorem}
\begin{proof}
Without loss of generality, assume that $f_0(0) = 0$.  In particular, this implies that 
\begin{equation}
||f_0||_{L^\infty(B_R)} \leq \Omega(R).
\end{equation}
Let $\phi\in C^{\infty}_c(B_1; [0,1])$ be a smooth cutoff function with $\phi\equiv 1$ on $B_{1/2}$.  Then for any $M>0$, we have that $f_0(x)\phi\left(\frac{x}{M}\right)\in W^{1,\infty}(\R^{d-1})$.  Furthermore whenever $|x|\leq M$ and $y\in \R^{d-1}$, we can bound the difference
\begin{equation}
\begin{split}
\bigg|f_0(x)\phi\left(\frac{x}{M}\right)- f_0(y)\phi\left(\frac{y}{M}\right)\bigg| &\leq |f_0(x)-f_0(y)| \phi\left(\frac{ y}{M}\right) + |f_0(x) |  \bigg|\phi\left(\frac{ x}{M}\right) - \phi\left(\frac{ y}{M}\right)\bigg| 
\\&\leq \Omega(|x-y|) +  \Omega(M) \min\left\{1, C\frac{ |x-y|}{M}\right\}.
\end{split}
\end{equation}
Thus as $f_0(x)\phi\left(\frac{ x}{M}\right) \equiv 0$ outside of $B_{M}$, we have that for all $x,y\in \R^{d-1}$
\begin{equation}\label{e:Omegaepsilon}
\bigg|f_0(x)\phi\left(\frac{x}{M}\right)- f_0(y)\phi\left(\frac{y}{M}\right)\bigg| \leq  \Omega(|x-y|) +  \Omega(M) \min\left\{1, C\frac{ |x-y|}{M}\right\} =: \Omega^{(M)}(|x-y|).
\end{equation}
Note that for any fixed $\delta>0$, we have that 
\begin{equation}
\int\limits_\delta^\infty \frac{\Omega^{(M)}(R)}{R^2}dR = \int\limits_\delta^\infty \frac{\Omega(R)}{R^2}dR + C\frac{\Omega(M)}{M}\left(1+\log\left(\frac{M}{C \delta}\right)\right).
\end{equation}
As $\Omega$ is integrably sublinear, we have that 
\begin{equation}
\int\limits_1^\infty \frac{\Omega(R)}{R^2}dR < \infty \qquad \Rightarrow \qquad \liminf\limits_{R\to \infty} \frac{\Omega(R)\log(R)}{R} = 0.
\end{equation}
Thus it follows that 
\begin{equation}
\liminf\limits_{M\to \infty} \int\limits_\delta^\infty \frac{\Omega^{(M)}(R)}{R^2}dR = \int\limits_\delta^\infty \frac{\Omega(R)}{R^2}dR .
\end{equation}
for any fixed $\delta>0$.  Furthermore, the same reasoning shows
\begin{equation}\label{e:epsilonslopebound}
\liminf\limits_{M\to \infty} ||\nabla_x (f_0(\cdot)\phi\left( \frac{\cdot}{M}\right))||_{L^\infty} \leq ||\nabla_x f_0||_{L^\infty} + \liminf\limits_{M\to \infty} C\frac{ \Omega(M)}{M} = ||\nabla_x f_0||_{L^\infty}.
\end{equation}

Now let $\eta\in C^{\infty}_c(B_1; [0,\infty))$ be a smooth mollifier with $\int \eta dx  = 1$.  For any $M>0$, let $\eta_M(x) =M^{2}\eta\left(Mx\right)$.  Define the smooth initial data 
\begin{equation}
f_0^{(M)}(x) =\left[f_0(\cdot)\phi\left( \frac{\cdot}{M}\right)\right]*\eta_M(x)  \in C^{\infty}_c(\R^{2}).
\end{equation}
Then by basic properties of mollifiers, \eqref{e:Omegaepsilon} and \eqref{e:epsilonslopebound} imply that 
\begin{equation}
|f^{(M)}_0(x) - f_0^{(M)}(y)|\leq \Omega^{(M)}(|x-y|), \qquad ||\nabla_x f_0^{(M)}||_{L^\infty} \leq ||\nabla_x \left[f_0(\cdot)\phi\left(\frac{\cdot}{M}\right)\right]||_{L^\infty}.
\end{equation}

Thus for any $M>0$, we have by Lemma \ref{l:smoothglobalexistence} that there exists a classical solution $f^{(M)}: [0,\infty)\times \R^{2}\to \R$ to the Muskat equation \eqref{e:fequation} with initial data $f_0^{(M)}$.  Passing to a subsequence $M_k\to \infty$ we can assume that 
\begin{equation}
||\nabla_x f_0^{(M_k)}||_{L^\infty} \to ||\nabla_x f_0||_{L^\infty}, \qquad \int\limits_{\delta}^\infty  \frac{\Omega^{(M_k)}(R)}{R^2}dR \to \int\limits_{\delta}^\infty \frac{\Omega(R)}{R^2}dR, \qquad \forall \delta>0.
\end{equation}
It thus follows by the a priori estimates of Sections \ref{s:modulusrho} and \ref{s:regularitytime} that the sequence $(f^{(M_k)}
)_{k}$ is uniformly bounded in $C([0,\infty)\times \R^2)\cap C^{1,\alpha}_{loc}((0,\infty)\times \R^2)\cap L^\infty_{loc}((0,\infty); C^{1,1}(\R^2))$.    Hence by passing to a further subsequence, we have that there exists a function $f:(0,\infty)\times \R^{2}$ solving the Muskat equation \eqref{e:fequation} such that 
\begin{equation}
f^{(M_k)} \to f \ \text{ in }C^{1}_{loc}((0,\infty)\times \R^{2}).
\end{equation}
Furthermore, this solution $f$ will satisfy all the previous a priori estimates.  By Proposition \ref{p:regularityspace} we have that for $k$ sufficiently large and $t<1/2$, 
\begin{equation}
|f_0^{(M_k)}(x) - f^{(M_k)}(t,x)|\leq Ct (-\log(t) + \int\limits_1^\infty \frac{\Omega(R)}{R^2}dR ), 
\end{equation}
it follows that 
\begin{equation}
\begin{split}
|f_0(x) - f(t,x)| &\leq Ct \left(-\log(t) + \int\limits_1^\infty \frac{\Omega(R)}{R^2}dR \right) + \liminf\limits_{k\to \infty } |f_0(x)-f_0^{(M_k)}(x)| + |f^{(M_k)}(t,x) - f(t,x)| 
\\&= Ct \left(-\log(t) + \int\limits_1^\infty \frac{\Omega(R)}{R^2}dR \right).
\end{split}
\end{equation}
Thus 
\begin{equation}
\lim\limits_{t\to 0} ||f_0(x) - f(t,x)||_{L^\infty} = 0,
\end{equation}
so $f:[0,\infty)\times \R^{2}\to \R$ is a solution to the Muskat problem with initial data $f_0$.  

The estimates proven in previous propositions are enough to guarantee that $f$ is a classical solution of the Muskat equation in the sense of definition \ref{d:classical}, with $||\nabla_x f||_{L^\infty([0,\infty)\times \R^{2})} = ||\nabla_x f_0||_{L^\infty(\R^{2})}<\displaystyle\frac{1}{\sqrt{5}}$.  So by Corollary \ref{c:uniqueness} $f$ is the unique such solution.  

\end{proof}

\section*{Acknowledgements}  
The author was partially supported by an NSF postdoctoral fellowship, NSF DMS 1902750.

\bibliography{Muskat-References}{}
\bibliographystyle{alpha}

\end{document}